\documentclass{amsart}

\usepackage[T1]{fontenc}
\usepackage{tikz-cd} 
\usepackage[a4paper,top=2.54cm,bottom=2.54cm,left=3cm,right=3cm,marginparwidth=1.75cm]{geometry}

\usepackage{amsmath,amssymb,amsfonts,amsthm,upgreek}
\usepackage{mathrsfs, mathtools}
\usepackage{graphicx}
\usepackage{slashed}
\usepackage[colorinlistoftodos,prependcaption, textsize=small, textwidth=2.5cm]{todonotes}
\usepackage[all]{xy}
\usepackage{cancel}
\usepackage{appendix}
\usepackage{enumitem}
\usepackage{multicol}
\usepackage[colorlinks=true, allcolors=blue,backref=page]{hyperref}
\usepackage[initials,alphabetic]{amsrefs}
\newcommand{\bN}{{\bf N}}

\newcommand{\cC}{\mathcal{C}}

\newcommand{\cK}{\mathcal{K}}
\newcommand{\cL}{\mathcal{L}}

\newcommand{\cP}{\mathcal{P}}

\newcommand{\slS}{\slashed{S}}

\newcommand{\N}{\bN}

\newcommand{\R}{\mathbb{R}}

\newcommand{\Spin}{{\rm Spin}}

\newcommand{\GL}{\mathrm{GL}}

\renewcommand{\epsilon}{\varepsilon}

\newcommand{\Ric}{{\rm Ric}}

\newcommand{\curl}{\mathop{\mathrm{curl}}}

\newcommand{\vol}{\mathrm{vol}}

\newcommand{\Sym}{\mathrm{Sym}}

\newcommand{\Bl}{\math\end{lemma}rm{Bl}}

\newcommand{\Div}{\mathop{\mathrm{div}}}

\def\<{\mathopen{}\left<}
\def\>{\right>\mathclose{}}
\def\({\mathopen{}\left(}
\def\){\right)\mathclose{}}

\usepackage{multicol, color}

\definecolor{gold}{rgb}{0.85,.66,0}
\definecolor{cherry}{rgb}{0.9,.1,.2}
\definecolor{burgundy}{rgb}{0.8,.2,.2}
\definecolor{orangered}{rgb}{0.85,.3,0}
\definecolor{orange}{rgb}{0.85,.4,0}
\definecolor{olive}{rgb}{.45,.4,0}
\definecolor{lime}{rgb}{.6,.9,0}
\definecolor{green}{rgb}{.2,.7,0}
\definecolor{grey}{rgb}{.4,.4,.2}
\definecolor{brown}{rgb}{.4,.3,.1}

\newtheorem{theorem}{Theorem}
\newtheorem{proposition}{Proposition}

\newtheorem{lemma}{Lemma}

\theoremstyle{remark}
\newtheorem{remark}{Remark}

\theoremstyle{definition}
\newtheorem{definition}{Definition}

\title{Nearly \texorpdfstring{$G_2$}{G2}-manifolds and \texorpdfstring{$G_2$}{G2}-Laplacian co-flows}
\author{Jason D. Lotay}
\author{Jakob R. Stein}
\date{\today}

\begin{document}

\begin{abstract}
Nearly $G_2$-structures define positive Einstein metrics in $7$ dimensions and are critical points, up to scale, for a geometric flow of co-closed $G_2$-structures with good analytic properties called the modified $G_2$-Laplacian co-flow.  We introduce a suitable normalization of this flow so that nearly $G_2$-structures are stable under rescaling.  However, we show that many nearly $G_2$-structures are unstable for this flow: specifically, all those naturally arising from 3-Sasakian geometry.  In particular, we demonstrate that the standard nearly $G_2$-structure on the round 7-sphere is an unstable critical point with high index.   
\end{abstract}
\maketitle
\setcounter{tocdepth}{2}
\tableofcontents
\section{Introduction}\label{sec:intro}

\subsection{Geometric flows and stability} In the study of geometric flows, one of the basic questions is whether a critical point for the flow is stable or not.  For example, unstable critical points typically pose a technical challenge for the flow process if one wants to find such critical points, whereas stable critical points motivate the pursuit of long-time existence and convergence results.  

Often, one is interested not just in critical points, but critical points up to rescaling: for example, if one starts the Ricci flow on a sphere, one might expect the sphere to shrink to a point, but after rescaling, the flow converges to the round sphere.  One way to achieve this rescaling is to normalize the flow: in the case of Ricci flow, the critical points become Einstein metrics with a given Einstein constant, rather than just Ricci-flat metrics.   It then becomes important to study the stability question for the critical points of this normalized geometric flow.  

This is the problem  we study in this article, in the context of \emph{nearly} (or \emph{nearly parallel}) $G_2$-structures (n$G_2$) on compact $7$-manifolds. These are critical points of suitably normalized geometric flows of co-closed $G_2$-structures, known as \emph{$G_2$-Laplacian co-flows}.  Since n$G_2$-structures define Einstein metrics with a positive Einstein constant, they are of significant interest in both $G_2$-geometry and Riemannian geometry more broadly.  Moreover, this relationship allows us to compare and contrast our results for $G_2$-Laplacian co-flows with well-known results for the Ricci flow.

\subsection{Main results} Our first key result identifies a necessary condition for stability of an n$G_2$-structure under a normalized $G_2$-Laplacian co-flow in terms of \emph{spectral data}: namely, the non-existence of eigenvalues of a certain first-order self-adjoint linear differential operator in a specified range.  We then use this criterion to show that even the round 7-sphere is \emph{unstable}, and is a high index critical point, which is in marked contrast to the situation for Ricci flow.  

Our second main result is to demonstrate that every n$G_2$-structure canonically arising from \emph{$3$-Sasakian geometry} is in fact \emph{unstable} for the same normalized $G_2$-Laplacian co-flow as above.   Moreover, we can explicitly identify a direction of instability, showing that each n$G_2$-structure has index at least one, as a critical point for the normalized flow. Again, this contrasts with Ricci flow, where one of the two canonical Einstein metrics on a 3-Sasakian 7-manifold is stable within the natural family of Riemannian metrics, whereas the other ``squashed'' Einstein metric is unstable.

Altogether, we demonstrate that $G_2$-Laplacian co-flows are not well-suited to the study of n$G_2$-structures, despite their promising geometric properties.  On the other hand, since n$G_2$-structures can model possible singularity formation in $G_2$-Laplacian co-flows, their instability suggests that one may be able to avoid such potential singularities through suitable perturbations.

\subsection{Nearly \texorpdfstring{$G_2$}{G2}-structures} A \emph{$G_2$-structure} on a 7-manifold $M$ is defined by a 3-form $\varphi$ satisfying a non-degeneracy condition: see \S \ref{sec:G2structures} for details.  Such a 3-form defines a metric $g$ and orientation on $M$, so we naturally have a Hodge star and thus Hodge dual $\psi$ of $\varphi$.  A $G_2$-structure is \emph{co-closed} if $d\psi=0$, which then means that 
\begin{equation}\label{eq:coclosed}
    d\varphi=\tau_0\psi+*\tau_3
\end{equation}
where $\tau_0$ is a function and $*\tau_3\in\Omega^4_{27}$, i.e.~$\tau_3$ is a 3-form satisfying $\tau_3\wedge\varphi=0$ and $\tau_3\wedge\psi=0$.  A special sub-class of the co-closed $G_2$-structures are the \emph{nearly} $G_2$-structures (n$G_2$) that satisfy
\begin{equation}\label{eq:nG2}
    d\varphi=\tau_0\psi
\end{equation}
where $\tau_0\in\R\setminus\{0\}$ is constant.  As mentioned above, the associated metric to an n$G_2$-stucture as in \eqref{eq:nG2} is Einstein with Einstein constant a positive multiple of $\tau_0^2$, and so n$G_2$-structures are of considerable importance. We exclude the \emph{torsion-free} case where $d\varphi=0$ and $d\psi=0$, since this leads to metrics with holonomy contained in $G_2$ (and thus Ricci-flat) which have a rather different character.

\subsection{\texorpdfstring{$G_2$}{G2}-Laplacian co-flows} Given that n$G_2$-structures are co-closed and define positive Einstein metrics, one is motivated to ask whether they are critical points of a (natural) geometric flow of co-closed $G_2$-structures, suitably normalized.  

One such candidate is the \emph{$G_2$-Laplacian co-flow} proposed in \cite{Karigiannis2012}\footnote{The flow originally presented in \cite{Karigiannis2012} has the opposite sign for the velocity of the flow, which turns out not to be the correct choice for analytic purposes.}:
\begin{equation}\label{eq:LCF}
    \frac{\partial\psi}{\partial t}=\Delta_{\varphi}\psi,%-\kappa^2\psi,
\end{equation}
where %$\kappa\in \R\setminus\{0\}$, 
$\psi$ is the Hodge dual of a $G_2$-structure $\varphi$, $d\psi=0$ (which is preserved by the flow) and $\Delta_{\varphi}$ is the Hodge Laplacian associated to $\varphi$.  We see that n$G_2$-structures as in \eqref{eq:nG2} are indeed critical points for \eqref{eq:LCF}, up to rescaling, and thus they are critical points for the \emph{normalized $G_2$-Laplacian co-flow}:
\begin{equation}\label{eq:nLCF}
    \frac{\partial\psi}{\partial t}=\Delta_{\varphi}\psi-\kappa^2\psi,
\end{equation}
where $\kappa\in\R\setminus\{0\}$ is chosen so that the n$G_2$ critical point will have $\tau_0=\kappa$ (up to sign).

%If one takes $\kappa=0$ then 
The flow \eqref{eq:LCF} turns out to be the positive gradient flow of the Hitchin volume functional \cites{Hitchin2001a,Hitchin2001} on a compact manifold within the cohomology class of the initial  value for $\psi$ \cite{Grigorian2013}. Its critical points, which are the torsion-free $G_2$-structures, are strict maxima modulo diffeomorphisms, so \eqref{eq:LCF} is geometrically well-motivated.  Unfortunately, \eqref{eq:LCF} is not yet known to be analytically well-behaved: in fact, short-time existence is still not known because the flow is not even weakly parabolic (i.e.~parabolic modulo the action of diffeomorphisms).

To address these analytic issues, a \emph{modified $G_2$-Laplacian co-flow} was introduced in \cite{Grigorian2013}:
\begin{equation}\label{eq:MLCF}
    \frac{\partial\psi}{\partial t}=\Delta_{\varphi}\psi+\frac{1}{2}d\big((5\kappa-7\tau_0)\varphi\big),
\end{equation}
where $\kappa\in\R$, $\psi$ is the Hodge dual to a co-closed $G_2$-structure $\varphi$ as before, and thus the function $\tau_0$ is given by \eqref{eq:coclosed}.  

It is shown in \cite{Grigorian2013} that \eqref{eq:MLCF} is now weakly parabolic and has short-time existence.  Torsion-free $G_2$-structures are still critical points, and they  are furthermore \emph{stable} for the flow \eqref{eq:MLCF} if one takes $\kappa=0$ by \cite{Vezzoni2020}. If one takes $\kappa\neq 0$ then one obtains an additional critical point given by an n$G_2$-structure, which is always unstable simply under rescaling \cite{Vezzoni2020}. 

Given these observations one is motivated to normalize the flow \eqref{eq:MLCF} for co-closed $G_2$-structures to account for rescaling as follows:
\begin{equation}\label{eq:nMLCF}
    \frac{\partial\psi}{\partial t}=\Delta_{\varphi}\psi+\frac{1}{2}d\big((5\gamma\kappa-7\tau_0)\varphi\big)+\frac{5}{2}(1-\gamma)\kappa^2\psi,
\end{equation}
where $\kappa\neq 0$, $\tau_0$ is as in \eqref{eq:coclosed} and  $\gamma>2$ is a real constant.   
 The normalization of \eqref{eq:MLCF} in \eqref{eq:nMLCF} is essentially the same as the normalization of \eqref{eq:LCF} in \eqref{eq:nLCF}: we have just used the freedom to modify the constant in \eqref{eq:MLCF} so that an n$G_2$ structure as in \eqref{eq:nG2} with $\tau_0=\kappa$ is a critical point of \eqref{eq:nMLCF}.  The choice $\gamma>2$, explained further in \S \ref{sec:MLCF},  guarantees that a n$G_2$-structure with $\tau_0=\kappa$ is a stable critical point of \eqref{eq:nMLCF} under rescalings of the $G_2$-structure.  
 % Note that t is a second n$G_2$ critical point where $\tau_0=(\gamma-1)\kappa$ of \eqref{eq:nMLCF} which is instead unstable under rescaling.
 \begin{remark} Note that setting $\gamma=1$ in \eqref{eq:nMLCF} recovers \eqref{eq:MLCF} and we may always choose $\kappa>0$ without loss of generality.
 \end{remark}

\subsection{Summary} In \S \ref{sec:preliminaries}, we briefly review $G_2$-structures and some key formulae. In \S \ref{sec:nG2} we obtain a useful decomposition of the exact 4-forms when $(M,\varphi)$ is compact n$G_2$:  these 4-forms describe the possible tangent directions for our $G_2$-Laplacian co-flows near the n$G_2$-structure.  This decomposition is achieved by studying properties of certain first order differential operators on compact n$G_2$ manifolds.\footnote{Related results can also be found in \cite{SoleFarre2024}.}

In \S \ref{sec:MLCF} we first present our stability criterion.

\begin{theorem}  \label{thm:lin.stab}
The flow \eqref{eq:nMLCF} is linearly stable at an n$G_2$-structure with $\tau_0=\kappa$ only if the operator $d*$ acting on the exact forms in $\Omega^4_{27}$ has no eigenvalues $\kappa\mu$ such that
\begin{equation}\label{eq:eigenvalue.bounds}
   % -1> \mu> -\tfrac{5}{2}(\gamma-1).
   -\tfrac{5}{2}(\gamma-1)<\mu<-1.
\end{equation}
\end{theorem}

\begin{remark} It is interesting to compare Theorem \ref{thm:lin.stab} with destabilising directions for the Einstein-Hilbert functional
at the Einstein metric induced by the n$G_2$-structure with $\tau_0=\kappa$, on metrics of fixed volume and constant scalar curvature,   cf. \cite{Koiso1980}*{Definition 2.7}, \cite{Wang2023}*{Definition 1.1}. By \cite{Semmelmann2024}*{Proposition 2.3}, the component of such directions in exact forms in $\Omega^4_{27}$ are  eigenvalues $\kappa\mu$ of the operator $d*$ satisfying: 
\begin{equation*}
    \tfrac{1}{2}> \mu> -1.
\end{equation*}
In particular, these are all stable directions of \eqref{eq:nMLCF}. 
\end{remark}

\noindent We then study this operator $d*$ on the standard homogeneous n$G_2$ structure for the constant curvature one $7$-sphere, which has $\tau_0=4$ in \eqref{eq:nG2}, and deduce the following theorem using representation theory, and our study of exact 4-forms in \S \ref{sec:nG2}.

\begin{theorem}\label{thm:index}
 The standard homogeneous n$G_2$-structure on the 7-sphere inducing the constant curvature one metric is unstable for \eqref{eq:nMLCF} with $\kappa=4$, with index of order at least $10^3$.    
\end{theorem}

Finally, in \S \ref{sec:3Sak}, we study the case of 3-Sasakian 7-manifolds which admit two natural 3-parameter families of co-closed $G_2$-structures $\varphi_{\pm}$, depending on a choice of $\pm 1$, with dual 4-forms $\psi_{\pm}$.  Given $\kappa>0$ there is a unique n$G_2$-structure with dual 4-form $\psi_{\pm}^{\kappa}$ within each family.  

We first show the following, which is almost immediate from work in \cite{KennonLotay2023}. 

\begin{theorem}\label{thm:3Sak.stab}
Given $\kappa>0$ and any initial condition in the family $\psi_{\pm}$, the flow \eqref{eq:nLCF} exists for all $t>0$ and converges to $\psi_{\pm}^{\kappa}$.  Hence, $\psi^{\kappa}_{\pm}$ is stable for \eqref{eq:nLCF} within the family $\psi_{\pm}$. 
\end{theorem}

\noindent In particular, for the 7-sphere this says that both the n$G_2$-structures giving the round metric and second ``squashed'' Einstein metric are stable for the normalized $G_2$-Laplacian co-flow, which is markedly different from Ricci flow.

In contrast, we have the following result for \eqref{eq:nMLCF}.

\begin{theorem}\label{thm:3Sak.instab}
 For any $\kappa>0$, the n$G_2$-structures $\psi_{\pm}^{\kappa}$ are unstable for \eqref{eq:nMLCF}.  Moreover, they have index $1$ in the $3$-parameter family $\psi_{\pm}$.   
\end{theorem}

\noindent We further identify the maximally destabilizing direction within $\psi_{\pm}$ for $\psi_{\pm}^{\kappa}$ and see that it is exact, it lies in $\Omega^4_{27}$ and satisfies the eigenvalue condition in \eqref{eq:eigenvalue.bounds} as we would expect.  These results are found by direct analysis of a nonlinear ODE system.

\begin{remark}
For the n$G_2$-structure inducing the ``squashed'' Einstein metric with dual 4-form $\psi_{+}^{\kappa}$, \cite{Semmelmann2024}*{Proposition 5.11} shows that one can recover a subspace of destabilising directions in Theorem \ref{thm:lin.stab} directly, via certain eigenfunctions of the scalar Laplacian, acting on a canonically-defined Riemannian 4-orbifold equipped with an anti-self-dual Einstein metric.
\end{remark}

\begin{remark}
Our results give some evidence for the fact that all n$G_2$-structures will be unstable for the modified $G_2$-Laplacian co-flow given in \eqref{eq:nMLCF}, which is potentially worthy of future study.  It certainly indicates that the modified $G_2$-Laplacian co-flow has some rather undesirable features when it comes to n$G_2$-structures.
\end{remark}

\noindent{\bf Acknowledgements:} JDL was partially supported by the National Science Foundation under Grant No.~DMS-1928930 and the Simons Foundation as a Simons Visiting Professor at SLMath, Berkeley, California, during the Fall Semester 2024 when this project was initiated. JRS was supported by the São Paulo Research Foundation (Fapesp) 2024/13414-2 linked to 2021/04065-6 BRIDGES.
%\subsection{Main Results and Plan of Paper} 
\section{Preliminaries} \label{sec:preliminaries}
\subsection{\texorpdfstring{$G_2$}{G2}-structures}
\label{sec:G2structures}
Let us begin with some preliminaries on $G_2$-structures (cf.~ \cites{Bryant2006,Dwivedi2023,Fernandez1982,Karigiannis2020}). 

\begin{definition}
    A 3-form $\varphi$ on a 7-manifold $M$ is a \emph{$G_2$-structure} if the following formula defines a Riemannian metric $g$ and a volume form $\vol$ on $M$\footnote{Note that the orientation convention in \eqref{eq:G2.structure} agrees with  \cite{Bryant2006}, but is opposite to the one in \cites{Alexandrov2012, Nagy2021}.}:
\begin{equation}\label{eq:G2.structure}
   g(X,Y) \vol= \tfrac{1}{6} (X \lrcorner\varphi )\wedge (Y \lrcorner \varphi )\wedge \varphi,    
\end{equation}
where $X,Y$ are tangent vectors on $M$.  These data also determine a 4-form $\psi := *\varphi$ with $\varphi \wedge \psi = 7 \vol$, so $|\varphi|^2=|\psi|^2=7$.  Note that the pointwise stabilizer of $\varphi$ in $\GL(7,\R)$ is isomorphic to the exceptional Lie group $G_2$. 
\end{definition} 

\begin{remark}
The requirement for $\varphi$ to define a $G_2$-structure is an open condition and \eqref{eq:G2.structure} can be viewed as a ``positivity'' condition.  The existence of a $G_2$-structure on $M^7$ is equivalent to $M$ admitting a spin structure (and thus an orientation).
\end{remark}

We shall be primarily interested in the following classes of $G_2$-structures.

\begin{definition}
A $G_2$-structure $\varphi$ on $M$ is \emph{co-closed} if its Hodge dual $\psi$ is closed, i.e.~$d\psi=0$.  

A \emph{nearly $G_2$-structure}\footnote{Nearly $G_2$-structures are also called nearly parallel $G_2$-structures in the literature.} (n$G_2$) $\varphi$ satisfies $d\varphi=\tau_0 \psi$ for $\tau_0\in\R$ with $\tau_0\neq 0$.  Note that an n$G_2$-structure is co-closed.  We say $(M,\varphi)$ is a nearly $G_2$-manifold (n$G_2$) if $\varphi$ is an n$G_2$-structure.  A nearly $G_2$-structure induces a Riemannian metric $g$ as in \eqref{eq:G2.structure} which is Einstein, in particular $\Ric = \tfrac{3}{8} \tau_0^2 g$, which implies $M$ is compact if $g$ is complete.  We shall therefore assume from now on that $M$ is compact.
\end{definition}

\begin{remark}
If one demands $d\varphi=\tau_0\psi$ for a function $\tau_0$, then it must necessarily be locally constant cf.~\cite{Dwivedi2023}*{Remark 2.3}, and can be fixed up to constant rescalings of the resulting metric if $M$ is connected. In the conventions used in this paper, the round metric on the unit sphere in $\R^8$ has constant sectional curvature $1$, scalar curvature $42$, and is induced by an n$G_2$-structure with $\tau_0=4$.   
\end{remark}

%Under the action of %the pointwise stabiliser 
%$G_2$, we have the associated 
\subsection{Differential forms and torsion} Given a $G_2$-structure, we have the following orthogonal splittings of the  2-forms and 3-forms, corresponding to irreducible $G_2$-representations, which will play a central role in our study.

\begin{definition} Let $\varphi$ be a $G_2$-structure on $M$.  We then have orthogonal decompositions
\begin{equation} \label{eq:g2forms}
\Omega^2 =   \Omega^2_7 \oplus  \Omega^2_{14}\quad\text{and}\quad \Omega^3 =   \Omega^3_1 \oplus  \Omega^3_7 \oplus \Omega^3_{27}, 
\end{equation}
where the subscripts indicate the ranks of the vector bundles of which the forms in question are sections.  Using the Hodge star, we also obtain splittings of $\Omega^4$ and $\Omega^5$:
\begin{equation} \label{eq:g2forms.star}
\Omega^4 =   \Omega^4_1 \oplus  \Omega^4_7 \oplus \Omega^4_{27}\quad\text{and}\quad \Omega^5 =   \Omega^5_7 \oplus  \Omega^5_{14}. 
\end{equation}
We shall let $\pi_j:\Omega^k\to\Omega^k_j$ denote the orthogonal projections.
\end{definition}

It will be useful for us to have explicit descriptions for the summands in \eqref{eq:g2forms}, and hence \eqref{eq:g2forms.star}, which we now recall.

\begin{lemma}\label{lem:forms.decomp} Let $\varphi$ be a $G_2$-structure on $M$ with Hodge dual $\psi$.  We may characterise the summands in \eqref{eq:g2forms} as:
\begin{gather} \label{eq:g2twoforms.7}
 \Omega^2_7 = \lbrace \alpha \in \Omega^2 \mid *(\alpha\wedge \varphi) = 2 \alpha \rbrace =\lbrace * (X \wedge \psi) \mid X \in \Omega^1 \rbrace;\\
\Omega^2_{14} = \lbrace \alpha \in \Omega^2 \mid *(\alpha\wedge \varphi) = -\alpha \rbrace = \lbrace \alpha \in \Omega^2 \mid \alpha\wedge \psi =0 \rbrace;\label{eq:g2twoforms.14}   \\
%Meanwhile, we can characterize the irreducible summands in $\Omega^3$ as:
%\begin{gather}
 \Omega^3_1 = \lbrace f \varphi \mid f \in \Omega^0 \rbrace;\qquad\qquad \Omega^3_{7} = \lbrace * (X \wedge \varphi) \mid X \in \Omega^1 \rbrace;\\
 \Omega^3_{27} = \lbrace \alpha \in \Omega^3 \mid \alpha\wedge \psi =0,\, \alpha \wedge \varphi=0 \rbrace.    
\end{gather}
\end{lemma}

 An immediate consequence of the decompositions \eqref{eq:g2forms} is that we can describe the \emph{torsion} of a $G_2$-structure as follows.

 \begin{definition}
    The \emph{torsion forms} of a $G_2$-structure $\varphi$ with Hodge dual $\psi$ are given by
    $\tau_0 \in \Omega^0$, $\tau_1 \in \Omega^1$,  $\tau_{2} \in \Omega^2_{14}$ and $\tau_{3} \in \Omega^3_{27}$ such that: 
\begin{align} \label{eq:g2torsion}
  d\varphi =\tau_0 * \varphi + 3 \tau_1 \wedge \varphi + *\tau_{3}& &d \psi = 4 \tau_1 \wedge * \varphi + * \tau_{2}
\end{align}
In particular, co-closed $G_2$-structures are those for which $\tau_1=\tau_2=0$ and n$G_2$-structures are those for which additionally $\tau_0$ is a non-zero constant. 
 \end{definition}

 \begin{remark}
A $G_2$-structure $\varphi$ is \emph{torsion-free} if all the torsion forms in \eqref{eq:g2torsion} vanish, i.e.~$d\varphi=d\psi=0$.  In this case the metric induced by $\varphi$ is Ricci-flat: in fact, it has holonomy contained in $G_2$.
 \end{remark}
% $\varphi$ splits into components $(\tau_0,\tau_1,\tau_2,\tau_3)$, where 
% With respect to the induced metric, we have:
% \begin{align} \label{eq:norms}
% |d \varphi|^2 = 7 \tau_1^2 + 12 |\tau_7|^2 + |\tau_{27}|^2& &|d *\varphi|^2 = 12 |\tau_7|^2 + |\tau_{14}|^2      
% \end{align}
To help compute the torsion forms, we recall the following basic lemma. 
\begin{lemma} \label{lem:g2torsion} The torsion form $\tau_0$ in \eqref{eq:g2torsion} is given by:
\begin{align*}
\tau_0 = \tfrac{1}{7} * ( d \varphi \wedge \varphi).    
\end{align*}
\end{lemma}
%We have decompositions of $\Omega^4$ and $\Omega^5$ by taking the Hodge star of the decompositions \eqref{eq:g2forms} above. Let $\pi_1$, $\pi_7$, $\pi_{27}$ denote the orthogonal projection of $\Omega^4$ to the relevant irreducible $G_2$-representation. Then we have the following:
We shall also need the following elementary result.
\begin{lemma} \label{lem:dtau3} For a $G_2$-structure $\varphi$ with Hodge dual $\psi$, we have that the torsion form $\tau_3$ in \eqref{eq:g2torsion} satisfies $\pi_1 ( d \tau_3) = \tfrac{1}{7} |\tau_3|^2 \psi$. 
\end{lemma}
\begin{proof} Let $\pi_1 ( d \tau_3) = \lambda \psi$ for some $\lambda$ to be determined. Now, we compute 
\begin{equation*}
7 \lambda \vol =  \pi_1 ( d \tau_3) \wedge \varphi = d \tau_3 \wedge \varphi = \tau_3 \wedge d \varphi = \tau_3 \wedge * \tau_3 = |\tau_3|^2 \vol\qedhere   
\end{equation*} 
\end{proof}
%We will be primarily interested in the following class of $G_2$-structures: 
%\begin{definition} A \textit{nearly} $G_2$-\textit{structure} (n$G_2$) is a $G_2$-structure $\varphi$ satisfying $d\varphi =\tau_0 * \varphi$, with $\tau_0 \neq 0$.  
%\end{definition}

\section{Analysis on nearly \texorpdfstring{$G_2$}{G2}-manifolds}\label{sec:nG2}

In this section we will have a \emph{compact} n$G_2$-manifold $(M,\varphi)$ with $d\varphi=\tau_0\psi$ and induced positive Einstein metric $g$. We will study certain first-order differential operators on $(M,\varphi)$ and show how these operators can be used to effectively decompose differential forms on $M$.  This will be invaluable for our analysis of $G_2$-Laplacian co-flows.

\subsection{The curl operator}  We begin with a $G_2$-analogue of the curl operator in 3 dimensions.
%We will now describe several differential operators and their properties on n$G_2$ manifolds. 

\begin{definition}\label{dfn:curl} We define the self-adjoint \emph{curl} operator $\curl{}: \Omega^1 \rightarrow \Omega^1$ on the compact n$G_2$ $(M,\varphi)$ by $\curl{X} := * (d X \wedge \psi) $.  We shall let $\sigma(\curl)$ denote the spectrum of $\curl$.
\end{definition}

We note some useful formulae for the curl operator from \cite{Dwivedi2023}*{Corollary 2.8}:
\begin{align} \label{eq:curl}
\curl{d f} = d^* \curl{X} =0\quad\text{and}\quad\curl{\curl{X}} = \Delta X - d d^* X + \tau_0 \curl{X},
\end{align}
where here and throughout $d^*$ is the adjoint of $d$ with respect to the metric and orientation determined by $\varphi$ and $\Delta$ is the Hodge Laplacian defined by $\varphi$. 

A fact we shall use several times is that on any compact Riemannian manifold we have the $L^2$-orthogonal splitting of one-forms: 
\begin{align} \label{eq:splitting1forms}
 \Omega^1 = d (\Omega^0)  \oplus \ker d^*   
\end{align}
It follows from \eqref{eq:curl} that $\curl$ acts orthogonally with respect to this splitting. In the n$G_2$-setting, the induced metric is positive Einstein, thus it admits no non-trivial harmonic one-forms by the classical result of Bochner. We deduce the following.

\begin{lemma} Given the splitting \eqref{eq:splitting1forms} on a compact n$G_2$, we have that
$\ker \curl = d (\Omega^0)$ and $\curl$ is invertible on $\ker d^*$.
\end{lemma}

To continue our analysis of the curl operator we shall need some formulae for the operator $d^*$. 
\begin{lemma}\cite{Dwivedi2023}*{Lemmma 2.11} \label{lemma:formulae} Let $(M, \varphi)$ be compact n$G_2$ with $d\varphi=\tau_0\psi$ and metric $g$. Then for any $f\in\Omega^0$ and $X\in \Omega^1$ we have: 
\begin{itemize}
\item $d^* (f \varphi) = - * (df \wedge \psi)$;
\item $d^* (f \psi) =  * (df \wedge \varphi) + \tau_0 f \varphi$;
    \item $d^* (X \wedge \varphi) = \tfrac{4}{7} (d^* X) \varphi + \tfrac{1}{4} * (( 2 \curl{X} + \tau_0 X) \wedge \varphi) + \tfrac{1}{2} i_\varphi ( (\cL_X g)_0 )$; 
    \item $d^* (X \wedge \psi) = \tfrac{3}{7} (d^* X) \psi + \tfrac{1}{4} ( 2 \curl{X} - 3 \tau_0 X) \wedge \varphi -\tfrac{1}{2} * i_\varphi ( (\cL_X g)_0 )$, 
\end{itemize} 
where $i_\varphi: \Sym^2_0 \rightarrow \Omega^3_{27}$ is the isomorphism defined in \cite{Dwivedi2023}*{Eq.2.17}, and $(\cL_X g)_0$ is the trace-free part of $\cL_X g$. 
\end{lemma} 

\subsection{Killing fields} We now wish to discuss Killing fields on $(M,\varphi)$ as 1-forms.

\begin{definition}
 Let $\Omega^1_\mathrm{Killing}$ denote the 1-forms  on the compact n$G_2$ $(M,\varphi)$ which are dual to Killing vector fields.  Now, using  that the induced metric on $M$ is Einstein, we can apply results of Lichnerowicz--Obata to identify the space of Killing vector fields with the Laplacian eigenspace:
\begin{equation}\label{eq:Killing}
\Omega^1_\mathrm{Killing} := \lbrace X\in \Omega^1 \mid d^*X=0, \, \Delta X = \tfrac{3}{4} \tau_0^2 X\rbrace.    
\end{equation}
\end{definition}

 We now show that the n$G_2$-structure gives a further decomposition of $\Omega^1_\mathrm{Killing}$  into eigenspaces for $\curl$, recalling Definition \ref{dfn:curl}.
 
\begin{lemma} \label{lem:killingvf}
    Let $(M,\varphi)$ be compact n$G_2$ with $d\varphi=\tau_0\psi$. If $\tau_0 \lambda \in \sigma(\curl)$ for $\lambda \neq 0$, then $\lambda^2 - \lambda\geq \tfrac{3}{4}$. Moreover, $\Omega^1_\mathrm{Killing}$ admits the $L^2$-orthogonal decomposition:
\begin{align*}
 \Omega^1_\mathrm{Killing} = \lbrace X\in \Omega^1  \mid \cL_X \psi = 0 \rbrace 
  \oplus \lbrace  X\in \Omega^1  \mid \cL_X \psi = -\tau_0 X \wedge \varphi \rbrace 
\end{align*}
into the eigenspaces for $\curl$ with eigenvalues $-\tfrac{1}{2}\tau_0$, $\tfrac{3}{2}\tau_0 \in \sigma(\curl)$ respectively.  
\end{lemma}
\begin{proof} Using the Lichnerowicz--Obata theorem, applied to \eqref{eq:curl}, it follows that the non-zero eigenvalues $\tau_0 \lambda$ of $\curl$ on $\Omega^1$ satisfy the inequality $\lambda^2-\lambda\geq \tfrac{3}{4}$. Moreover, the equality case is satisfied if and only if the corresponding one-form is Killing. Finally, we use Lemma \ref{lemma:formulae} to identify the corresponding eigenspaces as above. 
\end{proof}
% \begin{remark} One can use Lemma \ref{lem:killingvf} to show that $X \in \Omega^1$ preserves the n$G_2$-structure if and only if the two-tensor $(\nabla^c X) \in \Omega^2_{14}$, under the orthogonal splitting
% \begin{align*}
% \Omega^1 \otimes\Omega^1 = \Sym^2 \oplus \Omega^2_7 \oplus \Omega^2_{14}    
% \end{align*}
% where $\nabla^c$ is the canonical $G_2$-equivariant connection with skew-parallel torsion given in \cite[Equation 2.18]{Alexandrov2012}. 
% \end{remark}
\begin{remark}
The fact that a Killing field preserves an n$G_2$-structure on a compact 7-manifold if and only if its dual 1-form is a $-\tfrac{1}{2}\tau_0$-eigenform for curl is the content of \cite{Friedrich1997}*{Theorem 6.2}. 
\end{remark}

\begin{remark}
Applying Lemma \ref{lem:killingvf} to the standard Spin(7)-invariant n$G_2$-structure on the round 7-sphere $S^7$, one finds that $\tau_0=4$, the bound $\lambda^2-\lambda\geq \frac{3}{4}$ is saturated, and $\Omega^1_{\text{Killing}}$ is 28-dimensional (isomorphic to the Lie algebra of $O(8)$). In the decomposition of $\Omega^1_{\text{Killing}}$, the $-2$-eigenspace of curl is 21-dimensional (isomorphic to the Lie algebra of Spin(7)) and the $6$-eigenspace of curl is 7-dimensional.
\end{remark}

\begin{remark} One may identify $G_2$-structures inducing a given Riemannian metric with the space of real unit spinors (up to conjugation).  Under this identification, the unit sphere in the space of real Killing spinors $K\slS (M)$: 
\begin{align}
   K\slS (M):= \lbrace \eta \in \slS (M) \mid  \nabla_X \eta = - \tfrac{1}{8} \tau_0 X . \eta \rbrace
\end{align}
corresponds to all compatible n$G_2$-structures. Here, $X \in \Omega^1$ acts on the spinor bundle $\slS (M)$ via   Clifford multiplication, and we take the lift of the Levi-Civita connection to $\slS (M)$.

In particular, $\Omega^1$ defines another infinitesimal action on the space of Killing spinors $ \eta \in K\slS (M)$ via the Lie derivative:
\begin{align}
 \cL_X (\eta) &:= \nabla_X \eta- \tfrac{1}{4} (dX).\eta = - \tfrac{1}{8} ( \tau_0 X + 2 \curl{X}). \eta
\end{align}
 cf.~\cite{Dwivedi2023}*{\S 3.1}.  Notice that this action is trivial only for vector fields preserving the $G_2$-structure, rather than just Killing fields, by Lemma \ref{lem:killingvf}.  
\end{remark}

\subsection{The divergence operator} Another differential operator we shall need is the following.

\begin{definition}\label{dfn:div}   %$\Div$. This can be defined explicitly by picking some local orthonormal frame $e^i$, and letting $X \in \Omega^1$. 
We define the \emph{divergence} operator $\Div$ on $(M,\varphi)$ as the map\footnote{Note that our sign convention agrees with \cite{Dwivedi2023}, but is opposite to the convention in \cite{Nagy2021}.}:
\begin{equation}
\mathrm{div}:\Sym^2_0\to \Omega^1, \qquad \Div(h)(X) = \sum (\nabla_{e^i}h)(e^i, X),   
\end{equation}
where $e^i$ denotes a local orthonormal frame (for the metric $g$) and $X$ is a tangent vector.  The divergence has adjoint
   \begin{equation}\label{eq:div.star}
       \mathrm{div}^*:\Omega^1\to \Sym^2_0,\quad X\mapsto -\tfrac{1}{2}(\cL_Xg)_0.
   \end{equation}
Denote the space of \emph{conformal Killing fields} $\cC \subset \Omega^1$, i.e.~the kernel of $\Div^*$. %the map $\Omega^1 \to \Sym^2_0$ given by $X \mapsto (\cL_X g)_0$.  
Apart from the round sphere, $\cC$ is equal to the space of Killing vector fields $\Omega^1_\mathrm{Killing}$ (i.e.~$\cC \cap \ker(d^*)$) on compact Einstein manifolds. 
\end{definition}

We then have a useful decomposition result, which holds for any compact Riemannian manifold.
\begin{proposition}\label{prop:sym20}
    Let $(M, \varphi)$ be compact n$G_2$.  Then, using the notation of Definition \ref{dfn:div}, we have an $L^2$-orthogonal decomposition
    \begin{equation}\label{eq:sym20}
        \Sym^2_0 = \lbrace \tfrac{1}{2}(\cL_Xg)_0\mid X\in \cC^{\perp}  \rbrace  \oplus \lbrace h \in \Sym^2_0 \mid \mathrm{div} (h) =0 \rbrace.
    \end{equation}
\end{proposition}
\begin{proof}  This result follows directly from \cite{Besse2008}*{Lemma 4.57} but we explain the argument here.
%   We see that the map 
%   \begin{equation*}
%       \mathrm{div}:\Sym^2_0\to \Omega^1,\quad h\mapsto \mathrm{div}h
 %  \end{equation*}
 %  has adjoint
 %  \begin{equation*}
%       \mathrm{div}^*:\Omega^1\to \Sym^2_0,\quad X\mapsto -\tfrac{1}{2}(\cL_Xg)_0.
%   \end{equation*}
   The map $\mathrm{div}^*$ in \eqref{eq:div.star} has injective symbol and $\ker\mathrm{div}^*=\cC$.  Hence, we have the $L^2$-orthogonal splitting
   \begin{equation*}
       \Sym^2_0=\mathrm{div}^*(\Omega^1)\oplus \ker\mathrm{div}=\mathrm{div}^*(\cC^{\perp})\oplus\ker\mathrm{div}
   \end{equation*}
   as desired.
\end{proof}
We also have the following Weitzenbock formula for $\Div {\Div}^*$ on $\Omega^1$, valid on any Riemannian $n$-manifold, which will be useful throughout, cf.~\cite{Nagy2021}*{Equation 3.8}:
\begin{align} \label{eq:divdiv*}
\Div {\Div}^* X = \tfrac{1}{2} \Delta X + (\tfrac{n-2}{2n})d d^* X - \Ric X   .
\end{align}
For example, using \eqref{eq:divdiv*}, we can characterize exact elements in the kernel of ${\Div}^*$.
\begin{lemma} Let $(M, \varphi)$ be compact n$G_2$ with $d\varphi=\tau_0\psi$. Then $df \in \cC$ given in Definition \ref{dfn:div} if and only if $\Delta df = \tfrac{7}{16} \tau_0^2 df$.   
\end{lemma}
\begin{proof}
 In the n$G_2$ setting, the right-hand side of \eqref{eq:divdiv*} vanishes on the one-form $df$ if and only if $\Delta df = \tfrac{7}{16} \tau_0^2 df$, so the result follows using integration by parts on the left-hand side.   
\end{proof}
% \begin{lemma} Let $(M, \varphi)$ be n$G_2$ and $f$ be a smooth function on $M$. Then
% \begin{align*}
% \Delta (f \varphi) = ( \Delta f + \tau_0^2 f) \varphi + \tfrac{1}{2} \tau_0 * ( df \wedge \varphi).  
% \end{align*}
% \end{lemma}
\subsection{Exact 4-forms} Let us move on now to a decomposition of the space of exact 4-forms $\Omega_{\mathrm{exact}}^4$, since these will be the velocity vectors for our $G_2$-Laplacian co-flow. Firstly, we make the following definition.

\begin{definition} We let 
\begin{equation}\label{eq:K}
    \mathcal{K}:= \Omega^4_{7, \mathrm{closed}}= \lbrace X \wedge \varphi \mid X \in \Omega^1 ,  d ( X \wedge \varphi) = 0 \rbrace.
\end{equation}
(The latter description of $\mathcal{K}$ follows from Lemma \ref{lem:forms.decomp}.)
\end{definition}
It will be more useful to characterize $\mathcal{K}$ as follows. 
\begin{lemma}\label{lem:K.Killing} Let $(M, \varphi)$ be compact n$G_2$ with $d\varphi=\tau_0\psi$ and $X \in \Omega^1$. Then, using the notation in Definitions \ref{dfn:curl} and \eqref{eq:K},  the following are equivalent:
\begin{enumerate}[label=\normalfont{(\roman*)}]
    \item $X \wedge \varphi \in \mathcal{K}$
    \item $\curl{X} = \tfrac{3}{2} \tau_0 X$
    \item $\tau_0 X\wedge\varphi = d*(X\wedge\varphi)$
\end{enumerate}
\end{lemma} 
\begin{proof} Clearly (iii) implies (i) as $\tau_0\neq 0$. Using Lemma \ref{lemma:formulae}, we have that (iii) implies (ii), and (ii) implies (iii) by Lemma \ref{lem:killingvf}. Finally, (i) implies that $dX \wedge \varphi = \tau_0 X \wedge \psi$. By \eqref{eq:g2twoforms.14}, this gives $dX \in \Omega^2_7$, and moreover that $dX = \tfrac{\tau_0}{2} *( X \wedge \psi)$. In particular, this implies $d *(X \wedge \psi)=0$, so (ii) follows by Lemma \ref{lemma:formulae}.
% To prove the final statement, we note that by Lemma \ref{lemma:formulae}, infinitesimal automorphisms of the $G_2$-structure are Killing vector fields $Y$ with $\curl{Y} = - \tfrac{\tau_0}{2}Y$, in particular $Y \wedge \varphi$ is co-exact. 
\end{proof}

Now, a key result we shall require to study $\Omega^4_{\mathrm{exact}}$ is the following.  (Recall the isomorphism $i_{\varphi}:\Sym^2_0\to\Omega^3_{27}$ mentioned in Lemma \ref{lemma:formulae}.)

\begin{lemma}\label{lem:3forms} Let $(M,\varphi)$ be n$G_2$ with $d\varphi=\tau_0\psi$ and $\sigma\in\Omega^3$.  Then there exist unique $f\in\Omega^0$, $X\in\Omega^1$ and $\eta=i_{\varphi}(h)\in\Omega^3_{27}$ for $h\in \Sym^2_0$, such that
\begin{equation}\label{eq:sigma.decomp}
    \sigma=f\varphi+*(X\wedge\varphi)+\eta.
\end{equation}
Moreover, there are first order linear differential operators 
\begin{align*}
 d^{7}_{14}:\Omega^1\to \Omega^2_{14}& &d^{27}_{14}:\Omega^3_{27}\to\Omega^2_{14}& &d^{27}_{27}:\Omega^3_{27}\to\Omega^4_{27}  
\end{align*}
such that
\begin{align}\label{eq:d.sigma}
  d\sigma &= (\tau_0f+\tfrac{4}{7}d^*X)\psi+ (df+\tfrac{\tau_0}{4}X+\tfrac{1}{2}\curl X-\tfrac{1}{2}\mathrm{div} h)\wedge\varphi+*\tfrac{1}{2}i_{\varphi}(\cL_{X}g)_0+d^{27}_{27}\eta\\
  d^*\sigma &= *\big((-df+\tau_0X-\tfrac{2}{3}\curl X+\tfrac{2}{3}\mathrm{div} h)\wedge\psi\big)+d^7_{14}X+d^{27}_{14}\eta.\label{eq:d*.sigma}
\end{align}
\end{lemma}
\begin{proof}
    Equation \eqref{eq:d.sigma} follows immediately from the decomposition \eqref{eq:sigma.decomp} and Lemma \ref{lemma:formulae}.  Equation \eqref{eq:d*.sigma} is a consequence of \cite{Dwivedi2023}*{Lemma 2.12}.
\end{proof}

We have a useful immediate consequence of Lemma \ref{lem:3forms}.

\begin{lemma}\label{lem:Omega427.closed} On an n$G_2$-manifold we  have that the closed forms in $\Omega^4_{27}$ satisfy
\begin{equation}
    \Omega^4_{27,\mathrm{closed}}\subseteq\{*i_{\varphi}(h)\mid\mathrm{div}h=0\}.
\end{equation}
\end{lemma}
\begin{proof}
    This is immediate from \eqref{eq:d*.sigma}.
\end{proof}

%\begin{proposition} Let $(M, \varphi)$ be compact n$G_2$ and let $\cC$ be the space of conformal Killing fields. Then the map $X \mapsto (\cL_X g)_0$ induces an $L^2$-orthogonal splitting
%\begin{align*}
 %   \Sym^2_0 = \lbrace (\cL_Xg)_0\mid X\in \cC^\perp  \rbrace  \oplus \lbrace h \in \Sym^2_0 \mid \mathrm{div} (h) \in \cC \rbrace.
%\end{align*}     
%\end{proposition}

% \begin{lemma} Let $\eta = i_\varphi(h) \in \Omega^3_{27}$ satisfy $\mathrm{div} (h) \in \cC$. Then $\pi_{27}(d\eta)$ is exact.    
% \end{lemma}
To achieve our required result for $\Omega^4_{\mathrm{exact}}$ we shall also require the following decomposition of the 4-forms.
\begin{proposition} \cite{Dwivedi2023}*{Proposition 3.7} \label{prop:4forms.decomp} Let $(M, \varphi)$ be compact n$G_2$ and let $\mathcal{K}$ be as in \eqref{eq:K}. Then we have the following direct sum decomposition:  
\begin{equation}\label{eq:4forms.decomp}
 \Omega^4 = \cK \oplus d \Omega^3_1 \oplus d^* \Omega^5_7 \oplus \Omega^4_{27} 
 %& &\textcolor{red}{\Omega^4_\mathrm{exact} =  \cK \oplus d \Omega^3_1 \oplus \Omega^4_{27, \mathrm{exact}}\text{ -- This formula is not correct}} 
\end{equation} 
\end{proposition} 

\begin{remark}  The decomposition \eqref{eq:4forms.decomp} is not $L^2$-orthogonal only because the spaces $d^*\Omega^5_7$ and $\Omega^4_{27}$ are not.  This fact means that, unfortunately, the statement and proof of the decomposition of the exact 4-forms in \cite{Dwivedi2023}*{Proposition 3.7} contains an error, which we shall correct below.
\end{remark}

\begin{lemma}\label{lem:exact}
Let $(M,\varphi)$ be compact n$G_2$ with metric $g$ and recall the isomorphism $i_{\varphi}$ in Lemma \ref{lemma:formulae}. %and let $\cC$ be the space of conformal Killing fields.  
For all $X\in\Omega^1$ %orthogonal to 
%\begin{equation*}
%    \mathcal{W}=\lbrace W\in\Omega^1\mid \curl W=\tfrac{\tau_0}{2} W\rbrace
%\end{equation*} 
there exists $Z\in\Omega^1$ such that 
\begin{equation*}
d^*(Z\wedge\psi)+*\tfrac{1}{2}i_{\varphi}(\cL_Xg)_0\in d(\Omega^3_1\oplus\Omega^3_7).
\end{equation*}
\end{lemma}

\begin{proof}
We see from Lemma \ref{lemma:formulae} and \eqref{eq:d.sigma} that if $d\varphi=\tau_0\psi$ then for any $f\in\Omega^0$ and $Y,Z\in\Omega^1$ we have:
\begin{align*}
d(f\varphi+*(Y\wedge\varphi))&=(\tau_0f+\tfrac{4}{7}d^*Y)\psi+(df+\tfrac{1}{2}\curl Y+\tfrac{\tau_0}{4}Y)\wedge\varphi+*\tfrac{1}{2}i_{\varphi}(\cL_Yg)_0;\\ 
   d^*(Z\wedge\psi)+*\tfrac{1}{2}i_{\varphi}(\cL_Xg)_0&=\tfrac{3}{7} d^*Z\psi+(\tfrac{1}{2}\curl Z-\tfrac{3\tau_0}{4}Z)\wedge\varphi+*\tfrac{1}{2}i_{\varphi}(\cL_{X-Z}g)_0.
\end{align*}
Taking $Y=X-Z$, we see that 
\begin{equation*}
 d*(f\varphi+*(Y\wedge\varphi))= d*(f\psi+Y\wedge\varphi)=d^*(Z\wedge\psi)+*\tfrac{1}{2}i_{\varphi}(\cL_Xg)_0 
\end{equation*}
(which means that the right-hand side lies in $d(\Omega^3_1\oplus\Omega^3_7)$)  
if and only if
\begin{equation} \label{eq:decomp.error}
    -\tau_0f+d^*Z=\tfrac{4}{7}d^*X\quad\text{and}\quad -df+\curl Z-\tfrac{\tau_0}{2}Z=\tfrac{1}{2}\curl X+\tfrac{\tau_0}{4}X.
\end{equation}
We may rewrite \eqref{eq:decomp.error} by defining
\begin{equation}\label{eq:P}
\cP:\Omega^0\oplus\Omega^1\to\Omega^0\oplus\Omega^1,\quad \cP(f,Z):=(-\tau_0f+d^*Z,-df+\curl Z-\tfrac{\tau_0}{2}Z)
\end{equation}
giving the equation
\begin{equation}\label{eq:P.solve}
    \cP(f,Z)=(\tfrac{4}{7}d^*X,\tfrac{1}{2}\curl X+\tfrac{\tau_0}{4}X).
\end{equation}
The equation \eqref{eq:P.solve} may be solved if and only if the right-hand side is orthogonal to $\ker\cP^*$, where $\cP^*$ is the adjoint of $\cP$. Explicitly, this adjoint is given by
\begin{equation}
    \cP^*(u,W)=(-\tau_0u-d^*W,du+\curl W-\tfrac{\tau_0}{2}W).
\end{equation}
Suppose $(u,W)\in\ker\cP^*$. Then, since $d^*\curl W=0$ by \eqref{eq:curl}, we have
\begin{equation*}
\Delta u=-\tfrac{\tau_0^2}{2}u
\end{equation*}
and thus $u=0$ as $\tau_0\neq 0$.  Consequently,
\begin{equation*}
    \curl W=\tfrac{\tau_0}{2}W
\end{equation*}
which forces $W=0$ as well by the observation that there are no eigenvalues for $\curl$ in $(0,\tau_0)$ by Lemma \ref{lem:killingvf}.  

This implies that $\ker\cP^*=0$, so $\cP$ in \eqref{eq:P} is surjective (in fact, an isomorphism), so we can solve \eqref{eq:P.solve}, which gives an $f$ and $Z$ as required in \eqref{eq:decomp.error}.
%We see that the kernel of the operator acting on $Z$ is precisely $\mathcal{W}$, so we deduce that we can solve this equation for $Z$ if and only if for all $W\in \mathcal{W}$
%\begin{equation}
%    0=\langle \tfrac{1}{2}\curl X+\tfrac{\tau_0}{4}X,W\rangle_{L^2}=\tfrac{\tau_0}{2}\langle X,W\rangle_{L^2}, 
%\end{equation}
%which is precisely the assumption we made.
\end{proof}

At this point we may now correct \cite{Dwivedi2023}*{Proposition 3.7}.
\begin{proposition}\label{prop:exact.4forms.decomp} Let $(M,\varphi)$ be compact n$G_2$ and let $\Omega^4_{27,\mathrm{exact}}$ be the exact forms in $\Omega^4_{27}$.   
We then have the following $L^2$-orthogonal decomposition:
\begin{equation}\label{eq:exact.4forms.decomp}
 \Omega^4_{\mathrm{exact}}=d(\Omega^3_1\oplus\Omega^3_7)\oplus \Omega^4_{27,\mathrm{exact}}.   
\end{equation}
\end{proposition}
\begin{proof} 
By Propositions \ref{prop:sym20} and \ref{prop:4forms.decomp} we know that any $\rho\in\Omega^4$ can be written uniquely as
\begin{equation*}
    \rho=\kappa+d(f\varphi)+d^*(Y\wedge\psi)+*i_{\varphi}(\cL_{X}g)_0+*i_{\varphi}(h)
\end{equation*}
for $\kappa\in\mathcal{K}$, $f\in\Omega^0$, $Y\in \mathcal{K}^\perp$, $X\in\cC^{\perp}$  and $h\in \Sym^2_0$ with $\mathrm{div}h=0$.  By Lemma \ref{lem:exact}, there exists $Z\in\Omega^1$ and $\sigma\in\Omega^3_1\oplus\Omega^3_7$ such that
\begin{equation*}
    d\sigma=d^*(Z\wedge\psi)+*i_{\varphi}(\cL_Xg)_0.
\end{equation*}
We deduce that 
\begin{equation}\label{eq:rho.dsigma}
    \rho-d\sigma=\kappa+d(f\varphi)+d^*((Y-Z)\wedge\psi)+*i_{\varphi}(h).
\end{equation}
Now for any $V\in\Omega^1$, we have that
\begin{equation*}
    \langle *i_{\varphi}h,d^*(V\wedge \psi)\rangle_{L^2}=\langle d*i_{\varphi}h,V\wedge\psi\rangle_{L^2}=-\langle d^*i_{\varphi}(h),*(V\wedge\psi)\rangle_{L^2}=0
\end{equation*}
using \eqref{eq:d*.sigma}, and that $\mathrm{div}h=0$.  Since $\pi_{27}d(f\varphi)=0$ by \eqref{eq:d.sigma}, we deduce that the decomposition on the right-hand side of \eqref{eq:rho.dsigma} is $L^2$-orthogonal. A similar argument using \eqref{eq:d.sigma} shows that the decomposition \eqref{eq:exact.4forms.decomp} is orthogonal.  

Since the image of $d^*$ is orthogonal to the image of $d$, then if $\rho=d\gamma$ is exact, the $L^2$-orthogonality of the decomposition on the righ- hand side of \eqref{eq:rho.dsigma} implies that $d^*((Y-Z)\wedge\psi)=0$. Hence, we have
\begin{equation*}
    d\gamma-d\sigma=\kappa+d(f\varphi)+*i_{\varphi}(h).
\end{equation*}
As $\kappa$ is exact by Lemma \ref{lem:K.Killing}, we deduce that $*i_{\varphi}(h)$ is exact as claimed.  
\end{proof}

\subsection{The operator \texorpdfstring{$d*$}{d*} on exact 4-forms} With the splitting \eqref{eq:exact.4forms.decomp} of the exact 4-forms in place, for the remainder of this section, we will study the differential operator $d*:\Omega^4 \to\Omega^4$. 
\begin{proposition} \label{prop:d*splitting} Let $(M,\varphi)$ be compact n$G_2$.  
    The map $d*:\Omega^4_{\mathrm{exact}}\to\Omega^4_{\mathrm{exact}}$ preserves the $L^2$-orthogonal decomposition \eqref{eq:exact.4forms.decomp}.
\end{proposition}
\begin{proof}
Note first that $d*$ is self-adjoint on $\Omega^4_{\mathrm{exact}}$ since $*d^*=d*$ on 4-forms.

By Lemma \ref{lem:Omega427.closed}, if $\eta=*i_{\varphi}(h)\in\Omega^4_{27,\mathrm{exact}}$ then $\Div h=0$.  By \eqref{eq:d.sigma} we deduce that $d*\eta\in\Omega^4_{27}$ and it is manifestly exact.  Since $d*:\Omega^4_{27,\mathrm{exact}}\to\Omega^4_{27,\mathrm{exact}}$ and $d*$ is self-adjoint on $\Omega^4_{\mathrm{exact}}$, it must preserve the orthogonal complement of $\Omega^4_{27,\mathrm{exact}}$ in \eqref{eq:exact.4forms.decomp}.
\end{proof}
In light of the splitting \eqref{eq:exact.4forms.decomp}, in order to study the operator $d*$ on $d(\Omega^3_1\oplus\Omega^3_7)$, we prove the following lemma.  (Recall the curl operator from Definition \ref{dfn:curl}.)
\begin{lemma} \label{lem:dOmega17} Let $(M,\varphi)$ be compact n$G_2$ with $d\varphi=\tau_0\psi$. Then, for any $\eta \in d(\Omega^3_1\oplus\Omega^3_7)$, there is a unique $(f,X)\in \Omega^0 \oplus \Omega^1$ $L^2$-orthogonal to the finite-dimensional vector space 
\begin{align} \label{eq:kerd}
\lbrace(0,X) \mid \curl{X} = -\tfrac{\tau_0}{2} X \rbrace \oplus \lbrace (-\tfrac{\tau_0}{4} f, df) \mid \Delta f = \tfrac{7}{16} \tau_0^2 f \rbrace   \subset \Omega^0 \oplus \Omega^1 
\end{align}
such that $\eta = d (f \varphi + * (X \wedge \varphi))$. 
\end{lemma}
\begin{proof}
 Let $(f,X)\in \Omega^0 \oplus \Omega^1$. By \eqref{eq:d.sigma}, $d (f \varphi + * (X \wedge \varphi)) = 0$ if and only if:
 \begin{align} \label{eq:kerdproof}
 \tau_0 f = - \tfrac{4}{7} d^* X;& &df  = - \tfrac{\tau_0}{4}X -\tfrac{1}{2} \curl X;& &{\Div}^*X =0.
\end{align}
Using the decomposition \eqref{eq:splitting1forms}, write $X=du+Y$ for some exact one-form $du$, and co-closed one-form $Y$. Then \eqref{eq:kerdproof} implies $\curl{Y}=-\tfrac{\tau_0}{2} Y$, $df = -\tfrac{\tau_0}{4}du$, and $\Delta f = \tfrac{7}{16} \tau_0^2 f$.  
\end{proof}
\begin{remark} For any $u,v \in \Omega^0$, integration by parts shows that $(v,du)$ is orthogonal to \eqref{eq:kerd} if and only if $\tfrac{7}{4}\tau_0 u - v$ is orthogonal to the Laplacian eigenspace $\Delta f = \tfrac{7}{16} \tau_0^2 f$. 
\end{remark}
We are now in a position to prove the following fact about the spectrum of $d*$ acting on exact 4-forms, again using the curl operator from Definition \ref{dfn:curl}. 
\begin{proposition} \label{prop:Omega417} Let $(M,\varphi)$ be compact n$G_2$ with $d\varphi=\tau_0\psi$,  and let $\sigma(d*)$ denote the spectrum of $d*$ on $d(\Omega^3_1\oplus\Omega^3_7)$.  Suppose $\tau_0 \mu \in \sigma(d*)$. Then either $\mu\geq1$  or $\mu< - \tfrac{3}{4}$. Moreover,  if $\mu$ is not $-1$, then we have an $L^2$-orthogonal decomposition of the $\tau_0\mu$-eigenspace:
\begin{align}\label{eq:d*.eigenspace}
\lbrace f \in \Omega^0 \mid \Delta f = \tau_0^2 (\mu+\tfrac{1}{2})(\mu-1)f \rbrace \oplus \lbrace X \in \Omega^1 \mid \curl{X} = \tau_0 (\mu+\tfrac{1}{2}) X \rbrace ,
\end{align}
where the isomorphism takes $(f,X) \mapsto d \big( \tau_0 (\mu +\tfrac{1}{2})f \varphi + * ((df + X) \wedge \varphi)\big) \in d(\Omega^3_1\oplus\Omega^3_7)$. 

In the case $\mu = -1$, the second component in the decomposition \eqref{eq:d*.eigenspace} vanishes. 
\end{proposition}
\begin{proof}
By Lemma \ref{lem:dOmega17}, we can write any $\eta \in d(\Omega^3_1\oplus\Omega^3_7)$ uniquely using $(f,X) \in \Omega^0_1 \oplus \Omega^1_7$ $L^2$-orthogonal to the finite-dimensional space \eqref{eq:kerd}. We have from Lemma \ref{lem:3forms} and Definition \ref{dfn:div} that
\begin{align*}
\eta&=d(f\varphi+*(X\wedge\varphi)) \\
&= (\tau_0 f + \tfrac{4}{7} d^* X) \psi + (df + \tfrac{\tau_0}{4}X + \tfrac{1}{2} \curl(X) )\wedge \varphi - * i_\varphi ({\Div}^* X) \\
&=(\tau_0 f + d^* X) \psi + (df -\tfrac{\tau_0}{2}X + \curl(X) )\wedge \varphi- d^*( X\wedge \psi),
\end{align*}
where in the last line we have used Lemma \ref{lemma:formulae}. Now, we compute using \eqref{eq:d.sigma}:
\begin{align*}
    d*d(f\varphi+*(X\wedge\varphi))&=(\tfrac{4}{7}\Delta f+\tau_0^2f+\tfrac{5\tau_0}{7}d^*X)\psi\\
    &+ \big(\tfrac{5\tau_0}{4}df-\tfrac{\tau_0^2}{8}X+\tfrac{1}{2}\curl{\curl{X}}+dd^*X\big)\wedge\varphi \\
    &-*i_{\varphi}\big(\mathrm{div}^*(df-\tfrac{\tau_0}{2}X+\curl X) \big).
\end{align*}
The $\Omega^4_{27}$ component gives that an eigenform of the operator $d*$ has eigenvalue $\tau_0 \mu$ only if:
\begin{align} \label{eq:divX}
 \mathrm{div}^*(df-\tfrac{\tau_0}{2}X+\curl X) = \tau_0 \mu  \mathrm{div}^* X.  
\end{align}
Along with the formula \eqref{eq:divdiv*} for $\Div \Div^*$, and \eqref{eq:curl} for $\curl \curl$, we deduce that $d*d$ preserves the splitting of $\Omega^0 \oplus \Omega^1$ into eigenforms of the Laplacian. So, in the remaining equations, we can uniquely decompose 
\begin{align*}
 (f,X)= (v + C, du + Y) \in \Omega^0 \oplus \Omega^1   
\end{align*}
where $u,v \in \Omega^0$ such that $\int_M u  =\int_M v = 0$, $Y \in \Omega^1$ such that $d^*Y = 0$, and $C = \int_M f$ is a constant. Recalling Lemma \ref{lem:dOmega17}, we can also assume that $\curl(Y) \neq -\tfrac{\tau_0}{2} Y$, and if $u,v$ both satisfy $\Delta u =  \tfrac{7}{16} \tau_0^{2} u$, then $v=\tfrac{7}{4}\tau_0 u$. 

The eigenvalue problem for $d*$ then gives that $Y$ is a solution to 
\begin{align} \label{eq:curlY}
 \curl{\curl{Y}} - \tfrac{\tau_0^2}{4} Y = \mu (\tfrac{\tau_0^2}{2} Y + \tau_0 \curl{Y}).   
\end{align}
Since by assumption $\curl(Y) \neq -\tfrac{\tau_0}{2} Y$, \eqref{eq:curlY} factorizes to give a first order equation:
\begin{align} 
 \curl{Y} = \tau_0(\mu + \tfrac{1}{2})Y.    
\end{align}
By Lemma \ref{lem:killingvf}, this implies that $Y$ vanishes unless $|\mu|>1$, or $\mu = 1$.  

Meanwhile, $(v,u)$ satisfy the pair of equations:
\begin{align} \label{eq:deltau}
\tfrac{4}{7}\Delta v+\tau_0^2v+\tfrac{5\tau_0}{7} \Delta u = \tau_0 \mu (\tau_0 v +  \tfrac{4}{7} \Delta u );& &\tfrac{5\tau_0}{4}v-\tfrac{\tau_0^2}{8}u+\Delta u=\tau_0 \mu (v + \tfrac{\tau_0}{4}u);
\end{align}
and $\tau_0^2 C = \tau^2_0 \mu C$. 

Instead of solving the equations \eqref{eq:deltau} for $(v,u)$ simultaneously, we split finding eigenvalues for $d*$ of the form $\tau_0 \mu$ into cases (recalling Definition \ref{dfn:div}):
\begin{enumerate}[label=\normalfont{(\roman*)}]
     \item  Suppose $du, dv \in \cC^\perp$. Then $v = \tau_0 ( \mu + \tfrac{1}{2} ) u$ by \eqref{eq:divX}. From \eqref{eq:deltau}, we  deduce that
    \begin{align*}
      \Delta u = \tau_0^2 (\mu+\tfrac{1}{2})(\mu-1)u. 
    \end{align*}
    By Obata's theorem, we then have that $(v,u)$ vanishes unless $|\mu-\tfrac{1}{4}|>1$, with the strict inequality by the assumption $du, dv \in \cC^\perp$. 
    \item Suppose $du, dv \in \cC = \ker {\Div}^*$. Using \eqref{eq:divdiv*}, this implies both $(v,u)$ lie in the Laplacian eigenspace $\Delta w =  \tfrac{7}{16} \tau_0^{2} w$ for $w \in \Omega^0$. Thus \eqref{eq:deltau} becomes the linear equation for $(v,u)$:
   \begin{align*}
   (\tfrac{5}{4} - \mu) (v + \tfrac{\tau_0}{4} u ) = 0
   \end{align*}  
    By assumption, we must also have that $v=\tfrac{7}{4}\tau_0 u$, so we see that \eqref{eq:deltau} admits non-zero solutions if and only if $\mu = \tfrac{5}{4}$.  
    \item Suppose $\langle du,dv \rangle_{L^2} = 0$, then \eqref{eq:deltau} forces $\mu = \tfrac{5}{4}$, and that $du, dv \in \cC$ as above. 
\end{enumerate}
Combining the observations above completes the proof.
\end{proof}
\section{Stability of \texorpdfstring{$G_2$}{G2}-Laplacian co-flows and spectral data}\label{sec:MLCF}

We now turn to our primary interest, namely geometric flows of co-closed $G_2$-structures which have n$G_2$-structures as critical points. As we saw in the introduction (\S \ref{sec:intro}), there are two natural candidates for such flows: the \emph{normalized $G_2$-Laplacian co-flow} \eqref{eq:nLCF} and the \emph{normalized modified $G_2$-Laplacian co-flow} \eqref{eq:nMLCF}.  Since the former is not known to even have short-time existence we shall only discuss it briefly, and instead focus on the modified co-flow.  

The primary goal of this section is to derive a stability criterion for n$G_2$-structures under \eqref{eq:MLCF} involving the spectrum of the operator $d*$ acting on $\Omega^4_{27,\mathrm{exact}}$, as studied in \S \ref{sec:nG2}.  We shall then apply this result to demonstrate that the standard n$G_2$-structure for the round 7-sphere is highly unstable for the flow.

\subsection{Normalized \texorpdfstring{$G_2$}{G2}-Laplacian co-flows}
We may first consider the following (cf.~\cite{Karigiannis2012}).

\begin{definition}\label{dfn:nLCF} We define the \emph{normalized $G_2$-Laplacian co-flow} for a 1-parameter family $\psi(t)$ (depending on $t$) of Hodge duals of co-closed $G_2$-structures $\varphi(t)$ on a compact 7-manifold $M$ by:
\begin{equation}\label{eq:nLCF.2}
\frac{\partial\psi}{\partial t}=Q_0(\psi):=\Delta \psi-\kappa^2\psi\quad\text{and}\quad d\psi=0,
\end{equation}
for $\kappa\in\R$ with $\kappa\neq 0$, where $\Delta\psi=dd^*\psi=d*d\varphi$ (since $d\psi=0$).   We see that if $[\psi]=0\in H^4(M)$ initially (as is the case for n$G_2$-structures) then this is preserved by \eqref{eq:nLCF.2}.  
\end{definition} 

We also have the following elementary observation concerning n$G_2$-structures and \eqref{eq:nLCF.2}.

\begin{lemma}  
An n$G_2$-structure $\varphi$ with $d\varphi=\tau_0\psi$ is a critical point of \eqref{eq:nLCF.2} if and only if $\tau_0^2=\kappa^2$.
\end{lemma}

\begin{remark}\label{rmk:coflow.Hitchin}
    If we take $\kappa=0$ in \eqref{eq:nLCF.2} then it is demonstrated in \cite{Grigorian2013} that this flow is the positive gradient flow for the Hitchin volume function defined on $\psi\in [\psi(0)]$:
    \begin{equation}\label{eq:Hitchin.vol}
        \mathcal{V}(\psi)=\frac{1}{7}\int_M\varphi\wedge\psi.
    \end{equation}
It is also shown that critical points of the flow would then be torsion-free $G_2$-structures and are strict maxima of $\mathcal{V}$, modulo diffeomorphisms.   In this way, the $G_2$-Laplacian co-flow has many of the desirable geometric features that mirror those of the well-studied $G_2$-Laplacian flow of \emph{closed} $G_2$-structures introduced by Bryant, cf.~\cite{Bryant2006}.    
\end{remark}

 \begin{remark}\label{rmk:coflow.analysis}
It follows from \cite{Grigorian2013} that, unlike the Ricci flow or aforementioned $G_2$-Laplacian flow,  \eqref{eq:nLCF.2} is not even parabolic modulo diffeomorphisms.  As well as lacking a proof of short-time existence of \eqref{eq:nLCF.2}, the structure of the flow equation \eqref{eq:nLCF.2} means that, even if one assumes existence of the flow, results such as ``Shi-type'' estimates are also unavailable.  
\end{remark}
 
 Given the analytic issues with \eqref{eq:nLCF.2} raised in Remark \ref{rmk:coflow.analysis}, we are motivated to introduce a modification of the flow following \cite{Grigorian2013}.

%\subsection{Normalized modified \texorpdfstring{$G_2$}{G2}-Laplacian co-flow}

%Let $\psi_t$ be a one-parameter family of closed, non-degenerate $4$-forms on a $7$-manifold $M$, defining a one-parameter family of $G_2$-structures $(\varphi_t, \psi_t, g_t)$, and denote $\dot{\psi}:= \tfrac{\partial \psi_t}{\partial t} $. The modified co-flow introduced by Grigorian in \cite{Grigorian2013} can be written:
%\begin{align} \dot{\psi} = \Delta \psi + \tfrac{1}{2} d ( (5 \kappa - 7 \tau_0) \varphi) 
%\end{align}
%for some fixed constant $\kappa$. While n$G_2$-structures appear as critical points of the modified co-flow for $\tau_0 = \kappa$, Bedulli-Vezzoni show that these critical points are unstable, at fixed scale \cite{Vezzoni2020}. 

\begin{definition}\label{dfn:nMLCF} We define the \emph{normalized modified $G_2$-Laplacian co-flow} for a 1-parameter family $\psi(t)$ of Hodge duals of co-closed $G_2$-structures $\varphi(t)$ on a compact 7-manifold $M$ by:
    \begin{equation} \label{eq:flow} \frac{\partial\psi}{\partial t}=Q(\psi):= \Delta \psi + \tfrac{1}{2} d \big( (5 \gamma \kappa - 7 \tau_0) \varphi\big) + \tfrac{5}{2} (1-\gamma) \kappa^2 \psi\quad\text{and}\quad d\psi=0,
\end{equation}
for constants $\gamma,\kappa\in\R$ with $\kappa\neq 0$ and where $\tau_0$ is given by \eqref{eq:g2torsion}.  Note that taking $\gamma=1$ in \eqref{eq:flow} gives the modified $G_2$-Laplacian co-flow introduced in \cite{Grigorian2013}.  Again, the condition that $[\psi]=0\in H^4(M)$ is preserved by \eqref{eq:flow}, which is the situation of interest for us when studying n$G_2$-structures.
\end{definition}

\begin{remark}
A key motivation for \eqref{eq:flow} is that now, by \cite{Grigorian2013}, the flow is parabolic modulo diffeomorphisms.  It therefore has short-time existence and, moreover, enjoys ``Shi-type'' estimates by \cite{GaoChen2018b}.  Another desirable feature is that if one takes $\gamma=1$ and $\kappa=0$ in \eqref{eq:flow} then torsion-free $G_2$-structures are critical points, and they are dynamically stable under the flow \cite{Vezzoni2020}.
\end{remark}

To help motivate the precise form of the flow \eqref{eq:flow}, we now relate it to n$G_2$-structures.

\begin{lemma}\label{lem:nG2.nMLCF}
    A $G_2$-structure $\varphi$ with $d\varphi=\tau_0\psi$ is a critical point of \eqref{eq:flow} if and only if  $\tau_0=\kappa$ or $\tau_0=(\gamma-1)\kappa$.
\end{lemma}

\begin{proof}
If $d\varphi=\tau_0\psi$ then
\begin{equation*}
    Q(\psi)=-\tfrac{5}{2}(\tau_0-\kappa)(\tau_0-(\gamma-1)\kappa)\psi.
\end{equation*}
The result now follows from \eqref{eq:flow}.
\end{proof}

\begin{remark}
Lemma \ref{lem:nG2.nMLCF} shows that, for any $\gamma$, we have an n$G_2$ critical point for \eqref{eq:flow} with $\tau_0=\kappa$.  However, unless $\gamma=1$ or $\gamma=2$, we will always have a second n$G_2$ critical point. 
\end{remark}

As an aside, we look at the behaviour of the Hitchin volume functional under \eqref{eq:flow}, for which we recall the torsion forms in \eqref{eq:g2torsion}.

\begin{lemma} Along the flow \eqref{eq:flow}, the Hitchin volume functional $\mathcal{V}$ in \eqref{eq:Hitchin.vol} satisfies:
\begin{align*}
 \tfrac{d}{d t} \mathcal{V}= \tfrac{1}{4} \int_M \left( \tfrac{1}{7} | \tau_3|^2 - \tfrac{5}{2}( \tau_0 - \kappa)(\tau_0 - (\gamma-1) \kappa) \right) \vol 
\end{align*}
 In particular, $\mathcal{V}$ is monotonically increasing for $\kappa>0$ and $\gamma>1$ if $\kappa< \tau_0<(\gamma - 1)\kappa$.   
\end{lemma}
\begin{proof} Using e.g.~\cite{Grigorian2013}*{Equation 3.18}, we have that 
\begin{align}
\tfrac{d}{dt} \mathcal{V} = \tfrac{1}{4}\int_M \varphi \wedge \pi_1(\tfrac{\partial \psi}{\partial t}),  
\end{align}
so the statement is immediate from \eqref{eq:flow} and Lemma \ref{lem:dtau3}.
\end{proof}

\subsection{Stability} We now proceed to discuss our main interest concerning the normalized $G_2$-Laplacian co-flows we have introduced, and that is their behaviour near n$G_2$ critical points, particularly their (in)stability.
%However, we can easily fix the source of instability in \cite{Vezzoni2020} by normalizing the volume, to obtain the following flow: 
%\begin{align} \label{eq:flow} \dot{\psi} = Q(\psi):= \Delta \psi + \tfrac{1}{2} d ( (5 \gamma \kappa - 7 \tau_0) \varphi) + \tfrac{5}{2} (1-\gamma) \kappa^2 \psi
%\end{align}
%for some additional constant $\gamma$. We recover Grigorian's flow by setting $\gamma=1$ in \eqref{eq:flow}, and we still have n$G_2$-structures $(\varphi_0, \psi_0)$ with $d \varphi_0 = \kappa \psi_0$ appearing as critical points. 

The most obvious cause for the possible instability of an n$G_2$-structure along a geometric flow comes from rescaling the $G_2$-structure.  Our next result includes an extension of an observation in \cite{Vezzoni2020} for \eqref{eq:flow} for $\gamma=1$.  

\begin{lemma}\label{lem:scaling}
Suppose we are given an initial n$G_2$-structure $\varphi_0$ with $d\varphi_0=\kappa\psi_0$. Then $\psi_0$ is always stable under rescaling for \eqref{eq:nLCF.2} but is stable under rescaling for \eqref{eq:flow} if and only if $\gamma>2$.  Moreover, if $\gamma>2$ then the other n$G_2$ critical point of \eqref{eq:flow} is unstable.
\end{lemma}

\begin{proof}
As in \cite{Vezzoni2020} %for \eqref{eq:flow}, 
we make the ansatz, for some function $\mu(t)$:
\begin{equation}\label{eq:scaling}
\varphi = \mu^3 \varphi_0\quad\text{and}\quad\psi = \mu^4 \psi_0.     
\end{equation}
We then have that
\begin{equation}\label{eq:scaling.2}
d\varphi=\mu^3\kappa\psi_0=\mu^{-1}\kappa\psi    
\end{equation}
and so
\begin{equation}\label{eq:scaling.3}
\Delta\psi=d*d\varphi=\mu^{-2}\kappa^2\psi=\mu^2\kappa^2\psi_0.
\end{equation}

Putting this information in the normalized $G_2$-Laplacian co-flow \eqref{eq:nLCF.2} gives the ODE:
\begin{equation}\label{eq:scaling.ODE.2}
   \dot{\mu}=\tfrac{1}{4\mu}\kappa^2(1-\mu^2). 
\end{equation}
Clearly, the critical point corresponds to $\mu=1$ and is stable for \eqref{eq:scaling.ODE.2}.

If we instead insert the equations \eqref{eq:scaling}--\eqref{eq:scaling.3} into the modified co-flow \eqref{eq:flow}   we obtain the ODE:
\begin{align} \label{eq:scalingODE}
 \dot{\mu} = \tfrac{5}{8\mu}\kappa^2 ( \mu (1- \gamma) +1) (\mu-1)   
\end{align}
Since $$\left. \tfrac{\partial \dot{\mu}}{\partial \mu} \right|_{\mu =1} =  \tfrac{5}{8}\kappa^2 (2- \gamma)\quad\text{and}\quad \left. \tfrac{\partial \dot{\mu}}{\partial \mu} \right|_{\mu =(1-\gamma)^{-1}} =  \tfrac{5}{8}\kappa^2 (1-\gamma)(2- \gamma),$$ we see that $\mu =1$ is a stable critical point of \eqref{eq:scalingODE} if and only if $\gamma>2$, and that the other critical point with $\mu=\frac{1}{\gamma-1}$ is unstable in this case. 
%\begin{remark} Note that for a fixed $\kappa$, any n$G_2$ structure $\varphi, \psi$ such that $d \varphi= \tau_0 \psi$ has 
%\begin{align*}
% Q(\psi)= -\tfrac{5}{2}(\tau_0 - \kappa)( \tau_0 - (\gamma-1) \kappa) \psi  
%\end{align*}
\end{proof}

\begin{remark}
A corollary of Lemma \ref{lem:scaling} is the fact that the single n$G_2$ critical point of the original flow in \cite{Grigorian2013}, where $\gamma=1$ in \eqref{eq:flow}, is trivially unstable cf.~\cite{Vezzoni2020}.  We see that the case $\gamma=2$ in \eqref{eq:flow} is also special, since the two n$G_2$ critical points coincide, but then it is again trivially unstable.

Otherwise if $\gamma>1$, up to a constant rescaling of $\kappa$, we can always ensure that the n$G_2$-structure with $\tau_0=\kappa$ is the n$G_2$ critical point of \eqref{eq:flow} that is stable under (time-dependent) rescaling.  
\end{remark}
%Thus, when $\gamma \neq 2$ there is a second, unstable critical point, $\tau_0 = (\gamma-1) \kappa$, by the argument of \cite{Vezzoni2020}, corresponding to the critical point $\mu = \tfrac{1}{\gamma-1}$ of the ODE \eqref{eq:scalingODE}. In the case $\gamma = 2$, the (scaling) stable/unstable critical points coincide, but the critical point ends up being unstable.       
%\end{remark}
\begin{remark} Given Lemma \ref{lem:scaling}, from now on when discussing stability of n$G_2$-structures under \eqref{eq:flow} we will impose $\gamma>2$.
\end{remark}

Now that we have dealt with the trivial possible instability of our flow, we now want to tackle the full stability question.  To do this, we compute the linearisation of \eqref{eq:flow} at an n$G_2$-structure, recalling that it will be a flow of exact 4-forms if the initial condition is exact, along with the decomposition in Lemma \ref{lem:forms.decomp} and the curl operator in Definition \ref{dfn:curl}.  We also compute the linearisation of the normalized $G_2$-Laplacian co-flow \eqref{eq:nLCF.2} at an n$G_2$-structure, as this may be of independent interest. 

\begin{proposition} \label{prop:linearization} Let $(M,\varphi)$ be compact n$G_2$ with $d\varphi=\kappa\psi$ and let $\dot{\psi} = 4 \rho_0 \psi + \rho_1 \wedge \varphi - * \rho_3$ be an exact 4-form on $M$, where $\rho_0\in\Omega^0$, $\rho_1\in\Omega^1$ and $\rho_3\in\Omega^3_{27}$.  The linearisations of $Q_0$ in \eqref{eq:nLCF.2} and  $Q$ in \eqref{eq:flow} at $\psi$ in the direction of $\dot{\psi}$ %$Q(\psi) := \Delta \psi + \tfrac{1}{2} d ( (5 \gamma \kappa - 7 \tau_0) \varphi) + \tfrac{5}{2} (1-\gamma) \kappa^2 \psi$ 
%at the n$G_2$-structure $d \varphi = \kappa \psi$ on exact $4$-forms $\dot{\psi} = 4 f_0 \psi + f_1 \wedge \varphi - * f_3$ 
are given by: 
\begin{align}
D_\psi Q_0(\dot{\psi})&= -\cL_{V(\dot{\psi})}\psi-\Delta_{\psi}\dot{\psi}+2d((d^*\rho_1)\varphi)+\kappa d\big(\tfrac{3}{2}*\pi_1(\dot{\psi})-*\pi_7(\dot{\psi})-2*\pi_{27}(\dot{\psi})\big)-\kappa^2\dot{\psi};\label{eq:Q0.linearisation}\\
D_\psi Q (\dot{\psi}) &= - \cL_{V(\dot{\psi})} \psi - \Delta_\psi \dot{\psi} + \tfrac{\kappa}{2} d\big(  \lambda_1 * \pi_1 ( \dot{\psi}) + \lambda_7  * \pi_7 ( \dot{\psi}) + \lambda_{27} * \pi_{27} ( \dot{\psi})\big) + \tfrac{5}{2} (1-\gamma) \kappa^2 \dot{\psi}, \label{eq:Q.linearisation}
\end{align}
where $V(\dot{\psi}) = 7 d\rho_0 + 2 \curl{\rho_1}$ is the DeTurck vector field, and the constants $\lambda_1,\lambda_7,\lambda_{27}$ are
\begin{align*}
    \lambda_1 = \tfrac{15\gamma -2}{4},& &\lambda_7 = 5\gamma - 9,& &\lambda_{27}= 3 - 5 \gamma.
\end{align*}
\end{proposition}
\begin{proof}  Let $\psi_s$ be a 1-parameter family of Hodge duals of $G_2$-structures $\varphi_s$ with $\psi_0=\psi$ and $\frac{\partial}{\partial s}\psi_s|_{s=0}=\dot{\psi}$.  Throughout we shall use the dot notation to indicate $\frac{\partial}{\partial s}|_{s=0}$.

First, we note that, using e.g.~\cite{Bryant2006}*{Proposition 6},
\begin{align}\label{eq:dot.varphi}
\dot{\varphi}=%=\tfrac{\partial}{\partial s}\varphi_s|_{s=0} 
\dot{(*\psi)}= 3 \rho_0 \varphi + * ( \rho_1 \wedge \varphi) + \rho_3\quad\text{and hence}\quad *\dot{\psi} = \dot{\varphi}+\rho_0 \varphi - 2\rho_3. 
\end{align}
 Using \eqref{eq:g2torsion} and the fact that $\tau_0=\kappa$ and $\tau_3=0$ at $s=0$, we have
\begin{align}\dot{(d\varphi)}&=d\dot{ \varphi} = \dot{\left( \tau_0 \psi\right)} + \dot{\left(* \tau_3\right)}=\dot{\tau_0}\psi+\kappa\dot{\psi}+*\dot{\tau_3};\label{eq:d.dot.varphi}\\
\label{eq:dot.star.d.varphi}
\dot{(*d\varphi)}&=\dot{(\tau_0\varphi)}+\dot{\tau_3}=\dot{\tau_0}\varphi+\kappa\dot{\varphi}+\dot{\tau_3}.
\end{align}
From \eqref{eq:dot.varphi} and \eqref{eq:d.dot.varphi} we deduce that 
\begin{align}\label{eq:dot.tau3}
\dot{\tau_3} = *d \dot{\varphi} - \dot{\tau_0} \varphi - \kappa \left( \dot{\varphi} + \rho_0 \varphi - 2 \rho_3 \right). \end{align}
Note that Lemma \ref{lem:g2torsion} together with \eqref{eq:dot.varphi} and \eqref{eq:d.dot.varphi} gives
\begin{equation}\label{eq:dot.tau0}
    \dot{\tau}_0= \tfrac{4}{7} d^* \rho_1 -  \kappa \rho_0.
\end{equation}

We see that
\begin{equation*}
    \tfrac{\partial}{\partial s}|_{s=0}(\Delta_{\psi_s}\psi_s)=d\dot{(*d\varphi)},
    \end{equation*}
so if we let 
\begin{equation}\label{eq:p}
p(\dot{\psi}) := * d \dot{\varphi} - \kappa ( \rho_0 \varphi - 2 \rho_3 ),\end{equation} then \eqref{eq:dot.star.d.varphi} and \eqref{eq:dot.tau3} imply that 
\begin{equation}\label{eq:d.p} 
\tfrac{\partial}{\partial s}|_{s=0}(\Delta_{\psi_s} \psi_s) = d p ( \dot{\psi}).
\end{equation}
Using \eqref{eq:d.dot.varphi} and \eqref{eq:p} we may compute
\begin{align}
p(\dot{\psi}) + d^* \dot{\psi} &=  7  * ( d\rho_0 \wedge \varphi) + 2 d^* (\rho_1 \wedge \varphi) + 2 \kappa ( 3 \rho_0 \varphi + \rho_3) \nonumber\\
&= 7  * ( d\rho_0 \wedge \varphi) + 2 (d^*\rho_1) \varphi + 2 *\big(( \curl{\rho_1} - \tfrac{1}{2} \kappa \rho_1 ) \wedge \varphi\big) + 2 \kappa ( 3 \rho_0 \varphi + \rho_3)\nonumber\\ &\qquad + d ( * (\rho_1 \wedge \psi )) \nonumber\\
&= - V(\dot{\psi}) \lrcorner \psi + 2 (d^*\rho_1) \varphi -  * (\kappa \rho_1  \wedge \varphi) + 2 \kappa ( 3 \rho_0 \varphi + \rho_3) + d ( * (\rho_1 \wedge \psi ))\label{eq:p.d.star.psi}
\end{align}
where we have used Lemma \ref{lemma:formulae} in the second line and the stated formula for $V(\dot{\psi})$ in the final line.  The linearisation \eqref{eq:Q0.linearisation} of $Q_0$ follows immediately from \eqref{eq:p.d.star.psi}.
 
Similarly, if we let 
\begin{equation}\label{eq:q}
q(\dot{\psi}) := - \tfrac{1}{2} \big(( 7 - 5\gamma) \kappa \dot{\varphi} + 7 \dot{\tau_0} \varphi \big),
\end{equation}
we have that 
\begin{equation}\label{eq:d.q}
\tfrac{\partial}{\partial s}|_{s=0} \big(\tfrac{1}{2} d ( (5 \gamma \kappa - 7 \tau_0) \varphi_s)\big) = d q(\dot{\psi}).
\end{equation} 
Then we may compute \eqref{eq:q} using \eqref{eq:dot.varphi} and \eqref{eq:dot.tau0} to obtain 
\begin{align}\label{eq:q.2}
q(\dot{\psi}) =  -\tfrac{1}{2} \big( (4 d^* \rho_1 + (14 - 15 \gamma ) \rho_0 \kappa) \varphi + \kappa ( 7 - 5\gamma)( *  ( \rho_1 \wedge \varphi) + \rho_3 )\big).
\end{align}
Putting \eqref{eq:p.d.star.psi} and \eqref{eq:q.2} together, we get that
\begin{align}
p(\dot{\psi}) + q(\dot{\psi}) + d^* \dot{\psi} +  V(\dot{\psi}) \lrcorner \psi &=  \tfrac{\kappa}{2} \big( (15 \gamma - 2) \rho_0 \varphi +  ( 5 \gamma - 9) *  ( \rho_1 \wedge \varphi) + (5 \gamma -3) \rho_3 \big)\\ &\qquad + d ( * (\rho_1 \wedge \psi )) \nonumber\\
&= \tfrac{\kappa}{2} \left( \lambda_1 \pi_1 (* \dot{\psi}) +  \lambda_7 \pi_7 (* \dot{\psi}) + \lambda_{27} \pi_{27} (* \dot{\psi}) \right)\\ &\qquad + d ( * (\rho_1 \wedge \psi )).\label{eq:linearisation.final}
\end{align}
The linearisation \eqref{eq:Q.linearisation} now follows from \eqref{eq:flow}, \eqref{eq:d.p}, \eqref{eq:d.q} and \eqref{eq:linearisation.final}. 
%Setting $\tau_0 = \kappa$ gives the result. 
\end{proof}
Motivated by Proposition \ref{prop:linearization}, if we fix a co-closed $G_2$-structure on $M$ with Hodge dual $\psi_0$ then for any co-closed $G_2$-structure with Hodge dual $\psi\in[\psi_0]\in H^4(M)$ we may define the following operator:
\begin{equation}\label{eq:hatQ}
 \hat{Q}(\psi):= Q(\psi) + \cL_{V(\psi-\psi_0)} \psi,   
\end{equation}
where $Q$ is as in \eqref{eq:flow}.  Then the flow equation
\begin{equation*}
    \frac{\partial\psi}{\partial t}=\hat{Q}(\psi)
\end{equation*}
is equivalent up to the pullback by diffeomorphism to \eqref{eq:flow}, but now
kills the Lie derivative term appearing appearing in the linearized operator of Proposition \ref{prop:linearization} and thus  is now genuinely parabolic.
%
%We denote $\hat{Q}$ as the DeTurck flow corresponding to the operator $Q$, If we fix $\psi$ to lie in the co-homology class $[ \psi_0] \in H^4(M)$, for some non-degenerate choice of  representative $\psi_0$, then we can write the operator 
%\begin{align*}
% \hat{Q}(\psi):= Q(\psi) + \cL_{V(\psi-\psi_0)} \psi   
%\end{align*}
Moreover, by Proposition \ref{prop:d*splitting}, $D_\psi \hat{Q}$ maps $\Omega^4_{27,\mathrm{exact}}$ to itself, so we can consider the restriction to this subspace. 
\begin{lemma} \label{lemma:eigenvalue} Let $(M,\varphi)$ be compact n$G_2$ with 
$d\varphi=\kappa\psi$ and let $\hat{Q}$ be given by \eqref{eq:hatQ}. Then 
$D_\psi \hat{Q} (\dot{\psi})$ is negative-definite restricted to $\Omega^4_{27,\mathrm{exact}}$ if and only if $d *$ on $\Omega^4_{27,\mathrm{exact}}$ has no eigenvalues $\kappa \mu$ such that 
\begin{align*}
 %-1\geq\mu\geq -\tfrac{5}{2} (\gamma - 1).  
 -\tfrac{5}{2} (\gamma - 1)\leq\mu\leq -1.
\end{align*}  
Moreover, $\ker D_\psi\hat{Q}$ on $\Omega^4_{27,\mathrm{exact}}$ consists of the $\kappa\mu$-eigenspaces for $d*$ with $\mu\in\{-1,-\frac{5}{2}(\gamma-1)\}$, where the case $\mu=-1$ corresponds to first-order deformations of the n$G_2$-structure.
\end{lemma}
\begin{proof}  
% We split into cases, supposing $\dot{\psi} = \eta$ where $d*\eta=\kappa\mu\eta$. Now, we can rewrite:
%     \begin{align}
%         D_\psi\hat{Q}(\eta)= - \Delta_\psi \eta + \tfrac{\kappa}{2} d\big(  (\lambda_1-\lambda_{27}) * \pi_1 ( \eta) + (\lambda_7-\lambda_{27})  * \pi_7 ( \eta) + \lambda_{27} *  \eta \big) + \tfrac{5}{2} (1-\gamma) \kappa^2 \eta
%     \end{align}
% \begin{itemize}
%     \item First, suppose 
%     \begin{align*}
%      \eta = d (\kappa(\mu+\tfrac{1}{2})f \varphi + *(df \wedge \varphi) )   
%     \end{align*}
%     with $\Delta f = \kappa^2 (\mu+\tfrac{1}{2})(\mu-1)f$. Then
%     \begin{gather}   
%     \begin{aligned}
%      d * \pi_1(\eta)&= \tfrac{4}{7} \kappa^2(\mu+\tfrac{3}{4})(\mu+\tfrac{1}{2})d (f \varphi) \\
%      d * \pi_7(\eta)&= \kappa (\mu+\tfrac{3}{4}) \eta - \kappa^2(\mu+\tfrac{3}{4})(\mu+\tfrac{1}{2})d (f \varphi) 
%     \end{aligned}
%     \end{gather}
%     \item Next, suppose
%     \begin{align*}
%      \eta = d ( *(X \wedge \varphi) ) 
%     \end{align*}
%     with $\curl(X) = \kappa (\mu+\tfrac{1}{2})X$. Then
%     \begin{gather}   
%     \begin{aligned}
%      d * \pi_1(\eta)&= 0 \\
%      d * \pi_7(\eta)&= \tfrac{1}{2}\kappa (\mu+1) \eta
%     \end{aligned}
%     \end{gather}
%     So we have:
%     \begin{align}
%     D_\psi\hat{Q}(\eta)= - \kappa^2 (\mu+1)(\mu+\tfrac{1}{2})\eta    
%     \end{align}
%     By Lemma \ref{lem:killingvf}, we have $(\mu+1)(\mu+\tfrac{1}{2})>0$, so $D_\psi\hat{Q}$ is positive definite along $\eta$.
%     \item Finally, suppose $\eta \in \Omega^4_{27,\mathrm{exact}}$. Then
% \end{itemize}
Recall that $d*$ preserves $\Omega^4_{27,\mathrm{exact}}$ by Proposition \ref{prop:d*splitting} and so the space may be decomposed into eigenspaces for $d*$. 
If we set $\dot{\psi} = \eta$ where $d*\eta=\kappa\mu\eta$ then by Proposition \ref{prop:linearization} we have
\begin{align*}
\langle D_\psi \hat{Q} (\dot{\psi}), \dot{\psi} \rangle  = - \kappa^2 |\eta|^2 \left(\mu^2 - \tfrac{\mu}{2} \lambda_{27} - \tfrac{5}{2} (1- \gamma) \right) = - \kappa^2 |\eta|^2 ( \mu + 1) (\mu + \tfrac{5}{2} (\gamma - 1)) .
\end{align*}
The result follows from this formula together with \cite{Alexandrov2012}*{Theorem 3.5} to deduce the final remark in the statement.
\end{proof}

\noindent Lemma \ref{lemma:eigenvalue} together with Proposition \ref{prop:linearization} yields Theorem \ref{thm:lin.stab} in \S\ref{sec:intro}.

\begin{remark} Further computations using Proposition \ref{prop:Omega417} suggest that the eigenvalue bound in Lemma \ref{lemma:eigenvalue} is not sufficient to ensure $D_{\psi}\hat{Q}$ is negative definite on all of $\Omega^4_{\mathrm{exact}}$. In particular, it appears that one also requires a non-trivial lower bound on eigenvalues of the Laplacian acting on functions in terms of $\gamma$. 
\end{remark}

For the normalized $G_2$-Laplacian co-flow \eqref{eq:nLCF.2}, again fixing some reference co-closed $G_2$-structure defined by $\psi_0$ and considering $\psi\in[\psi_0]\in H^4(M)$, we may define an analogous operator $\hat{Q}_0$ to $\hat{Q}$ in \eqref{eq:hatQ} as follows:
\begin{equation}\label{eq:hatQ0}
    \hat{Q}_0(\psi):=Q_0(\psi)+\cL_{V(\psi-\psi_0)}\psi=\Delta\psi-\kappa^2\psi+\cL_{V(\psi-\psi_0)}\psi.
\end{equation}
In this way, the flow generated by $\hat{Q}_0$ is equivalent, modulo diffeomorphisms, to \eqref{eq:nLCF.2} but it is \emph{not} parabolic due to the presence of the $2d((d^*\rho_1)\varphi)$ term in \eqref{eq:Q0.linearisation}.  That said, we may still look at the question of stability in the direction of $\Omega^4_{27,\mathrm{exact}}$ for \eqref{eq:nLCF.2} to contrast it with Lemma \ref{lemma:eigenvalue}.

\begin{lemma}\label{lem:LCF.lin.stab}
    Let $(M,\varphi)$ be compact n$G_2$-structure with 
$d\varphi=\kappa\psi$ and let $\hat{Q}_0$ be given by \eqref{eq:hatQ0}. Then $D_\psi\hat{Q}_0$ is non-positive on $\Omega^4_{27,\mathrm{exact}}$.  Moreover, it is negative definite if and only if $-\kappa$ is not an eigenvalue of $d*$ on $\Omega^4_{27,\mathrm{exact}}$, which means that there are no first-order n$G_2$ deformations of $\psi$ in $\Omega^4_{27,\mathrm{exact}}$.
\end{lemma}

\begin{proof}
    If we let $\dot{\psi}=\eta$ be a $\kappa\mu$-eigenform of $d*$ on $\Omega^4_{27,\mathrm{exact}}$ then Proposition \ref{prop:linearization} implies that
    \begin{equation*}
        \langle D_{\psi}\hat{Q}_0(\dot{\psi}),\dot{\psi}\rangle=-\kappa^2(\mu+1)^2|\eta|^2,
    \end{equation*}
    which yields the result.
\end{proof}

\begin{remark} In fact, using Proposition \ref{prop:Omega417}, one can compute that $D_{\psi}\hat{Q}_0(\dot{\psi})$ is non-positive on all of $\Omega^4_{\mathrm{exact}}$, with kernel given by the first-order n$G_2$ deformations in $\Omega^4_{27,\mathrm{exact}}$ as in Lemma \ref{lem:LCF.lin.stab}. In particular, rigid n$G_2$-structures are always linearly stable under the normalised $G_2$-Laplacian co-flow \eqref{eq:nLCF.2}, up to diffeomorphisms.  
\end{remark}

\subsection{Instability of the round sphere} Following standard notation, we define
the \emph{index} of a critical point $\psi$ for the flow \eqref{eq:flow} as the dimension of the space of unstable directions at $\psi$ (modulo diffeomorphisms).  By Proposition \ref{prop:linearization}, the index is the dimension of the subspace on which $D_{\psi}\hat{Q}$, where $\hat{Q}$ is given in \eqref{eq:hatQ}, is positive definite. By Lemma \ref{lemma:eigenvalue}, for an n$G_2$ critical point with $d\varphi=\kappa\psi$, we see that the index is at least the total multiplicity of eigenspaces for $d*$ on $\Omega^4_{27,\mathrm{exact}}$ with eigenvalues $\kappa\mu$ satisfying $-1>\mu> -\tfrac{5}{2} (\gamma - 1)$. 

Given these observations, we will find a lower bound on this index for the round sphere, fixing the normalization $\tau_0=\kappa=4$ so that the metric has constant unit sectional curvature. 

\begin{theorem} \label{thm:eigenvalues} Let $S^7$ be equipped with the standard $\Spin(7)$-invariant $nG_2$-structure $\varphi$ with $d\varphi=4\psi$, inducing the round constant curvature $1$ metric. The spectrum $\sigma(d*)$ of $d*$ on $\Omega^4_{27, \mathrm{exact}}$ is contained in $\lbrace\pm (4+ \ell) \mid \ell \in \N \rbrace$. The multiplicity $\dim(l)$ of the eigenvalue $- (4+ \ell)$ is bounded below by:
\begin{align*}
 \dim(l)\geq\tfrac{1}{120} \tfrac{l^2 + 5l -16}{(l+4)(l+6)}\tfrac{(l+7)!}{(l+1)!} 
\end{align*}
\end{theorem}
\begin{proof}
Calculating the dimension exactly would involve working with the automorphism group $\Spin(7)$ of $\varphi$, and a lengthy computation of the relevant first-order differential operators. For simplicity, we will compute eigenvalues using the representation theory of the full isometry group $O(8)$, as this has already appeared in the literature e.g.~\cites{Bar2019,Folland1989}. For the aid of the reader, we recall the details of this here. 

Firstly, observe that the spectrum of $d*$ on $\Omega^4(S^7)$ must be symmetric about $0$, since the orientation-reversing isometry given by rotation by angle $\pi$ around a  line through the origin in $\mathbb{R}^8$ exchanges the eigenspaces of $d*$ with opposite sign, cf.~\cite{Bar2019}*{Theorem 3.3}. The same argument also shows that the multiplicity of each non-zero eigenvalue $\lambda$ is exactly half of the multiplicity of the eigenvalue $\lambda^2$ of the Laplacian on $\Omega^4_{\mathrm{exact}}(S^7)$. 

Using this fact, \cite{Bar2019}*{Theorem 5.2} shows that the negative part of the spectrum on $\Omega^4$  consists of eigenvalues of the form $- \left(4 + l\right)$ for non-negative integers  $l\in \N$, with multiplicity $d$ given by
\begin{align}\label{eq:dim.calc.1}
 d=\tfrac{(l+7)!}{(3!)^2 (l+4) l!}   
\end{align}
However, these isometries do not preserve the decomposition \eqref{eq:g2forms.star}. Instead, using the splitting in Proposition \ref{prop:d*splitting}, to find $\dim(l)$, we have to subtract from $d$ the dimensions of the spaces 
\begin{align}\label{eq:dim.calc.2}
\lbrace f \in \Omega^0 \mid \Delta f = (l+2)(l+8)f \rbrace\quad\text{and}\quad \lbrace X \in \Omega^1 \mid \curl{X} = - ( l+ 2) X \rbrace  
\end{align}
coming from eigenvalues of $d*$ in $d(\Omega^3_1\oplus\Omega^3_7)$ by Proposition \ref{prop:Omega417}. 

The first of the spaces in \eqref{eq:dim.calc.2} is spanned by the restriction to $S^7$ of homogeneous harmonic polynomials in $\R^8$ of degree $l+2$. Its dimension $d_0$ is given by the Weyl character formula \cite{Fulton1991}*{Equation 24.41}:
\begin{align}\label{eq:dim.calc.3}
 d_0 = \tfrac{2}{6!}\tfrac{(l+7)! (l+5)}{(l+2)!}   
\end{align}

We obtain an estimate for the dimension of the second space in \eqref{eq:dim.calc.2} using \eqref{eq:curl}, by observing that it is contained in the $(l+2)(l+6)$-eigenspace for the Laplacian on co-closed one-forms. This latter space is the irreducible $O(8)$-representation of highest weight contained in $\Sym^l(V) \otimes \Lambda^2(V)$, where $V$ is the standard representation, see \cite{Folland1989}*{Theorem C}\footnote{In the notation of \cite{Fulton1991}, this is the irreducible $SO(8)$-representation of highest weight $(l+1)L_1 + L_2$.}. The dimension $d_1$ of this representation is given by the Weyl character formula as before:
\begin{align}\label{eq:dim.calc.4}
 d_1 = \tfrac{2}{5!}\tfrac{(l+7)! (l+4)}{l!(l+2)(l+6)}   
\end{align}

Since $\dim(l)\geq d - d_0 - d_1$ by Proposition \ref{prop:d*splitting}, we have the estimate as claimed by combining \eqref{eq:dim.calc.1}, \eqref{eq:dim.calc.3} and \eqref{eq:dim.calc.4}. 
\end{proof}
% \begin{remark} For $l\leq2$ the inequality of Theorem \ref{thm:eigenvalues} is vacuous. However, from Proposition \ref{prop:Omega417}, we have the improved estimate $\dim(l)\geq d - d_1 = 64$ in the special case $l=1$.
% \end{remark}

We now return to the linearised modified co-flow operator $D_\psi Q$ of \eqref{eq:flow} of the standard homogeneous n$G_2$ round sphere with $\kappa=4$. By Lemma \ref{lemma:eigenvalue}, $D_\psi Q$ is positive definite on all the eigenspaces of $d*$ in $\Omega^4_{27,\mathrm{exact}}$ with eigenvalues between $-10$ and $-4$ ($-10$ included), for any $\gamma>2$. Thus, Theorem \ref{thm:eigenvalues} gives a lower bound of $7047$ on the index of $D_\psi Q$ by taking the sum of $\dim(l)$ for $3\leq l \leq 6$\footnote{Note that the estimate in Theorem \ref{thm:eigenvalues} is vacuous when $l =1,2$.}. This proves Theorem \ref{thm:index} in \S\ref{sec:intro} and contrasts dramatically with the situation for the normalized $G_2$-Laplacian co-flow \eqref{eq:nLCF.2}, where elements in $\Omega^4_{27,\mathrm{exact}}$ cannot contribute to the index for any n$G_2$-structure by Lemma \ref{lem:LCF.lin.stab}.

% $7047$ 

\section{3-Sasakian case}\label{sec:3Sak}
In this section, we consider the normalized modified $G_2$-Laplacian co-flow \eqref{eq:flow} for a natural family of co-closed $G_2$-structures on 7-manifolds admitting a 3-Sasakian metric.  We discover that all of the n$G_2$-structures which arise as critical points for the flow are unstable within this family.  Moreover, we can identify the unstable directions explicitly, which  shows that the full 3-parameter families are required to see the instabilities.  We also briefly compare this with the normalized $G_2$-Laplacian co-flow \eqref{eq:nLCF.2}, where we instead obtain long-time existence and convergence to the n$G_2$ critical point.

\subsection{\texorpdfstring{$G_2$}{G2}-structures} We start by briefly recalling the set-up (cf.~\cite{Agricola2010}).  If $M^7$ is orientable and admits a 3-Sasakian metric $g$, there is a locally free action of $SU(2)$ on $M$ whose leaf space is a 4-dimensional orientable orbifold $N$. and there is an Einstein metric $h$ on $N$ so that the quotient map $\pi:(M,g)\to (N,h)$ is a Riemannian submersion.  Moreover, we have an orthogonal decomposition $TM=\mathcal{V}\oplus\mathcal{H}$ where $\mathcal{V},\mathcal{H}$ are the vertical and horizontal distributions with respect to $\pi$. 

There is a global orthonormal basis $\{\eta_1,\eta_2,\eta_3\}$ of $\mathcal{V}^*$ and an orthogonal basis $\{\omega_1,\omega_2,\omega_3\}$ of $\Lambda^2_+\mathcal{H}^*$ such that
\begin{equation*}
    \omega_i\wedge\omega_j=2\delta_{ij}\vol_{N},
\end{equation*}
where $\vol_N$ is the volume form on $N$ pulled back to $\mathcal{H}$, and $\{\eta_1,\eta_2,\eta_3\}$  satisfy the structure equations for $(i,j,k)$ cyclic permutations of $(1,2,3)$:
\begin{align}\label{eq:3Sak.structure}
    d\eta_i&=-2\eta_j\wedge\eta_k-2\omega_i.
\end{align}
Given this data, we may then define two distinct families of co-closed $G_2$-structures $(\varphi_+, \psi_+)$, $(\varphi_-, \psi_-)$ on $M$, depending on constants $a,b,c\in\R^+$ as follows:
\begin{gather} \label{eq:3Sak.G2str.psi}
\begin{aligned}
\varphi_\pm&= \pm a^2b\eta_1\wedge\eta_2\wedge\eta_3-ac^2(\eta_1\wedge\omega_1+\eta_2\wedge\omega_2) \mp bc^2\eta_3\wedge\omega_3;\\
\psi_\pm&=c^4\vol_N \mp abc^2(\eta_2\wedge\eta_3\wedge\omega_1+\eta_3\wedge\eta_1\wedge\omega_2)-a^2c^2\eta_1\wedge\eta_2\wedge\omega_3.
\end{aligned}
\end{gather}
\begin{remark}
Of course, one can easily generalize these two 3-parameter families into 4-parameter families cf.~\cite{KennonLotay2023}*{Lemma 2.5}, but we shall not require this here. 
\end{remark}
\begin{remark}
The $G_2$-structures corresponding to $\psi_+$ and $\psi_-$ cannot be homotopic through $G_2$-structures by \cite{Crowley2015a}*{Examples 1.14 and 1.15}    
\end{remark}
The metrics $g_\pm$ and volume forms $\vol_{\pm}$ associated to the  $G_2$-structures in \eqref{eq:3Sak.G2str.psi} are given by 
\begin{equation*}
    g_+ = g_- =a^2(\eta_1^2+\eta_2^2)+b^2\eta_3^2+c^2g_N \quad\text{and}\quad \vol_{\pm}=\pm  a^2bc^4\eta_1\wedge\eta_2\wedge\eta_3\wedge\vol_N,
\end{equation*}
where $g_N$ is the pullback of the orbifold metric on $N$ to $\mathcal{H}$.
\subsection{Critical points} From \eqref{eq:3Sak.structure}--\eqref{eq:3Sak.G2str.psi} one can easily compute the following (cf.~\cite{KennonLotay2023}*{Lemma 2.14 and Lemma 3.1}):
\begin{align}
    d\varphi_{\pm}&=4(2a + \epsilon b)c^2\vol_N - 2 \epsilon b(a^2+c^2)(\eta_2\wedge\eta_3\wedge\omega_1+\eta_3\wedge\eta_1\wedge\omega_2)\nonumber\\
    &\quad - 2\epsilon (a^2b+2\epsilon ac^2-bc^2)\eta_1\wedge\eta_2\wedge\omega_3;\label{eq:3Sak.dphi}\\
    \tau_0&=\frac{4}{7}\frac{4a(a^2+c^2)+\epsilon b(2a^2-c^2)}{a^2c^2};\label{eq:3Sak.tau0}\\
    \Delta \psi_{\pm}
    &=8\left(2a^2+b^2+2c^2+\frac{2\epsilon bc^2}{a}-\frac{b^2c^2}{a^2}\right)\vol_N\nonumber\\
&\quad    -4\left(b^2+\frac{4\epsilon a^3b}{c^2}+\frac{2a^2b^2}{c^2}+\frac{2\epsilon bc^2}{a}-\frac{b^2c^2}{a^2}\right)(\eta_2\wedge\eta_3\wedge\omega_1+\eta_3\wedge\eta_1\wedge\omega_2)\nonumber\\
&\quad -4\left(2a^2-b^2+2c^2+\frac{4\epsilon a^3b}{c^2} +\frac{2a^2b^2}{c^2}-\frac{2\epsilon bc^2}{a} +\frac{b^2c^2}{a^2}\right)\eta_1\wedge\eta_2\wedge\omega_3,\label{eq:3Sak.Delta}
\end{align}
where $\epsilon = \pm 1$ respectively. As a direct consequence, we may now characterize the n$G_2$ critical points for the normalized Laplacian coflow and its modification within the family of $G_2$-structures above.

\begin{lemma}\label{lem:3Sak.nG2} 
    The ansatz \eqref{eq:3Sak.G2str.psi} satisfies $d\varphi=\kappa \psi$ for $\kappa>0$, if and only if
  \begin{equation*}
      \psi=\psi_{+}, \, a=b=\frac{1}{\sqrt 5}c= \frac{12}{5\kappa} \quad or \quad  \psi=\psi_{-}, \, a=b=c=\frac{4}{\kappa}
  \end{equation*}
  We denote the resulting $G_2$-structure by $\psi_{\pm}^\kappa$.
\end{lemma}

\begin{proof}
    This follows immediately from \eqref{eq:3Sak.G2str.psi}, \eqref{eq:3Sak.tau0} and \eqref{eq:3Sak.dphi}.
\end{proof}

\subsection{Normalized \texorpdfstring{$G_2$}{G2}-Laplacian co-flow} We first consider the normalized $G_2$-Laplacian co-flow as in Definition \ref{dfn:nLCF}: 
\begin{equation}\label{eq:3Sak.nLCF}
    \frac{\partial\psi}{\partial t}=\Delta_{\psi}\psi-\kappa^2\psi.
\end{equation}
   We first see from \eqref{eq:3Sak.Delta} that the ansatz  \eqref{eq:3Sak.G2str.psi} is preserved along the flow \eqref{eq:3Sak.nLCF}. 
Indeed, we have that \eqref{eq:3Sak.nLCF} is equivalent to the following ODE system   for functions $a,b,c$ of $t$ defining $G_2$-structures as in \eqref{eq:3Sak.G2str.psi}:
\begin{gather} \label{eq:3Sak.nLCF.1}
\begin{aligned}
\frac{d}{d t}(c^4) & = 8\left(2a^2 +b^2+ 2c^2 +\frac{2\epsilon bc^2}{a} - \frac{b^2c^2}{a^2}\right)-\kappa^2c^4; \\
\frac{d}{d t}(a^2c^2)& = 4\left(2a^2-b^2+2c^2+\frac{4\epsilon a^3b}{c^2}+\frac{2a^2b^2}{c^2}-\frac{2\epsilon bc^2}{a}   + \frac{b^2c^2}{a^2}\right)-\kappa^2a^2c^2; \\  
\frac{d }{d t}(abc^2) &= 4\left(\epsilon b^2+\frac{4a^3b}{c^2} + \frac{2\epsilon a^2b^2}{c^2} + \frac{2 bc^2}{a} -  \frac{\epsilon b^2c^2}{a^2}\right)-\kappa^2abc^2,
\end{aligned}
\end{gather}
where $\epsilon = \pm 1$. We then have the following result, which follows almost immediately from \cite{KennonLotay2023}*{Theorem 1.1}, and proves Theorem \ref{thm:3Sak.stab} in \S\ref{sec:intro}.

\begin{theorem}
Given $\kappa>0$, and initial condition $\psi(0) = \psi_{\pm}$ as in \eqref{eq:3Sak.G2str.psi}, the flow \eqref{eq:3Sak.nLCF} exists for all forward time, and converges to $\psi_{\pm}^\kappa$.  In particular,  $\psi_{\pm}^\kappa$ is stable for \eqref{eq:3Sak.nLCF} in the family \eqref{eq:3Sak.G2str.psi}.
\end{theorem}

\begin{proof}
Note that if $c(0)>0$ at the initial time $t=0$, then \eqref{eq:3Sak.nLCF.1} implies $c$ is positive for all forward times $t\geq 0$ for which $(a,b,c)$ exist. Thus, we can re-parameterize the system  \eqref{eq:3Sak.nLCF.1} using a new variable $s$ such that
    \begin{equation}\label{eq:3Sak.s}
        \frac{d s}{d t}=\frac{1}{c^2}
    \end{equation}
Furthermore, since $c>0$, if we define the scale-invariant quantities: 
    \begin{equation}\label{eq:3Sak.XY}
        X=\frac{a^2}{c^2}\quad\text{and}\quad Y=\frac{ab}{c^2},
    \end{equation}
    then \eqref{eq:3Sak.nLCF.1} gives that $X$ and $Y$ satisfy
    \begin{gather}
    \begin{aligned} \label{eq:3Sak.Xflow}
        \frac{d X}{d s}&=\frac{4}{X^2}\left((X+1)Y^2 + 2\epsilon(2X^2-2X-1)XY-2X^2(2X-1)(X+1)\right)\\
        \frac{d Y}{d s}&=\frac{4Y}{X^2}\left(2(1-X)Y^2+\epsilon(2X^2-3X-1)Y+2X(1-2X)\right).
    \end{aligned}
    \end{gather}
 These are the same equations as in \cite{KennonLotay2023}*{Lemma 3.2} for the (un-normalized) $G_2$-Laplacian co-flow, i.e.~setting $\kappa=0$ in \eqref{eq:3Sak.nLCF}. The quadrant $X>0$, $Y>0$ is preserved under this system, and the non-degenerate critical points are given by:
 \begin{align*}
    X=Y= \tfrac{1}{5}, \,\text{if }\epsilon=+1& &\text{and}& &X=Y=1, \, \text{if }\epsilon=-1
 \end{align*}
 corresponding to the n$G_2$-structures in Lemma \ref{lem:3Sak.nG2}, cf.~\cite{KennonLotay2023}*{Lemma 3.3}. 
 
 The proof of \cite{KennonLotay2023}*{Theorem 3.8} shows that given any non-degenerate initial value for $a,b,c$, which then determines an initial value $X>0$, $Y>0$, the flow \eqref{eq:3Sak.Xflow} will converge to the corresponding critical point. 
 
 We now observe from \eqref{eq:3Sak.nLCF.1} that
 \begin{equation} 
     \frac{d}{d t}(c^2)=4\left(2X+\frac{Y^2}{X}+2+2\epsilon\frac{Y}{X}-\frac{Y^2}{X^2}\right)-\frac{\kappa^2}{2}c^2.\label{eq:3Sak.cf.c}
 \end{equation}
Moreover, \cite{KennonLotay2023}*{Theorem 3.8} shows that $X,Y$ are bounded and bounded away from zero, so it follows from \eqref{eq:3Sak.cf.c} that $c$ must also remain bounded. Since $c>0$, we deduce that the variable $s$ is defined for all  time. We finally deduce from the convergence of \eqref{eq:3Sak.Xflow} the claimed convergence of the normalized Laplacian co-flow.
\end{proof}

\subsection{Normalized modified \texorpdfstring{$G_2$}{G2}-Laplacian co-flow} 
We now move on to the normalized modified $G_2$-Laplacian co-flow as in Definition \ref{dfn:nMLCF}:
\begin{equation}\label{eq:3Sak.nMLCF}
    \frac{\partial\psi}{\partial t}=\Delta_{\psi}\psi+\frac{1}{2}d\big((5\gamma\kappa-7\tau_0)\varphi\big)+\frac{5}{2}(1-\gamma)\kappa^2\psi
\end{equation}
for fixed $\kappa>0$ and $\gamma>2$.  We recall the following from Lemma \ref{lem:scaling}.

\begin{lemma} \label{lem:scalingODE} The n$G_2$-structures $\psi^\kappa_{\pm}$, $\psi_\pm^{\kappa (\gamma-1)}$ are both critical points for \eqref{eq:3Sak.nMLCF}. However, $\psi^\kappa_{\pm}$ is stable under rescaling the $G_2$-structure, whereas $\psi_\pm^{\kappa (\gamma-1)}$ is unstable.
\end{lemma}

Combining the ansatz \eqref{eq:3Sak.G2str.psi} with \eqref{eq:3Sak.dphi}--\eqref{eq:3Sak.Delta} we see that the flow \eqref{eq:3Sak.nMLCF} is equivalent to the following ODE system:
\begin{gather} \label{eq:3Sak.nMLCF.0}
\begin{aligned}
    \frac{d}{d t}(c^4)&=-48a^2-8b^2-48c^2 +10\epsilon\gamma\kappa bc^2+20\gamma\kappa ac^2-64\epsilon ab +\frac{5}{2}(1-\gamma)\kappa^2c^4;\\
    \frac{d}{dt}(abc^2)&=-\frac{8bc^2}{a}+5\gamma\kappa a^2b+5\gamma\kappa bc^2-32 ab+\frac{5}{2}(1-\gamma)\kappa^2 abc^2;\\
    \frac{d}{dt}(a^2c^2)&=-24a^2+8b^2-24c^2+\frac{16\epsilon bc^2}{a}+5\epsilon\gamma\kappa a^2b+10\gamma\kappa ac^2-5\epsilon\gamma\kappa bc^2-16\epsilon ab\\
    &\quad+\frac{5}{2}(1-\gamma)\kappa^2a^2c^2.
\end{aligned}
\end{gather}
\begin{remark} One can recover Lemma \ref{lem:scalingODE} directly from \eqref{eq:3Sak.nMLCF.0} by considering the following one-dimensional subsystems, cf.~\eqref{eq:scalingODE}: 
 \begin{equation*}
      \psi=\psi_{+}, \, a=b=\frac{1}{\sqrt 5}c \quad or \quad  \psi=\psi_{-}, \, a=b=c.
  \end{equation*}
\end{remark}
\begin{remark}
It follows from \eqref{eq:3Sak.nMLCF.0} that the subfamily of $\psi_{\pm}$ in \eqref{eq:3Sak.G2str.psi} with $a=b$ is preserved by  \eqref{eq:3Sak.nMLCF} if and only if $\psi = \psi_{+}$.  This is consistent with what occurs for the usual normalized $G_2$-Laplacian co-flow \eqref{eq:3Sak.nLCF}.   However, one explicitly sees the lack of homogeneity present in the modified ODE system \eqref{eq:3Sak.nMLCF.0} versus the corresponding normalized $G_2$-Laplacian co-flow system \eqref{eq:3Sak.nLCF.1}. 
\end{remark}

Since the n$G_2$-structure $\psi_\pm^{\kappa (\gamma-1)}$ is unstable under the flow a priori, we restrict to studying the stability of the critical point $\psi^\kappa_{\pm}$ under the normalized modified $G_2$-Laplacian co-flow.
\begin{theorem}\label{thm:3Sak.instab.2}
 The n$G_2$-structure $\psi^\kappa_{\pm}$ is an unstable critical point for \eqref{eq:3Sak.nMLCF}, with index one within the family of $G_2$-structures \eqref{eq:3Sak.G2str.psi}. The unstable subspace for the linearized flow at $\psi^\kappa_{\pm}$ within this family is spanned by $\Psi_{\pm} \in \Omega^4_{27,\mathrm{exact}}$, given explicitly by
 \begin{equation}\label{eq:3Sak.Psi+}
     \Psi_{+}=\eta_2\wedge\eta_3\wedge\omega_1+\eta_3\wedge\eta_1\wedge\omega_2-2\eta_1\wedge\eta_2\wedge\omega_3
 \end{equation}
% with eigenvalue $(\frac{5}{3}\gamma-\frac{25}{9})\kappa^2$
 and
 \begin{equation}\label{eq:3Sak.Psi-}
     \Psi_{-}=2\vol_N- (\eta_2\wedge\eta_3\wedge\omega_1+\eta_3\wedge\eta_1\wedge\omega_2+\eta_1\wedge\eta_2\wedge\omega_3).
 \end{equation}
 Moreover,
 \begin{equation}\label{eq:3Sak:eigenvalues}
     d*\Psi_+=-\tfrac{5}{3}\kappa\Psi_+\quad\text{and}\quad 
     d*\Psi_-=-\tfrac{3}{2}\kappa\Psi_-.
 \end{equation}
% with eigenvalue $(\frac{5}{4}\gamma-2)\kappa^2$.
\end{theorem}
\begin{remark} In particular, this shows  $\psi^\kappa_{+}$ is stable within the sub-family $\psi_{+}$ with $a=b$, even though it is unstable in the full family.  We also confirm in \eqref{eq:3Sak:eigenvalues} that the destabilizing forms satisfy the eigenvalue condition identified in Lemma \ref{lemma:eigenvalue}.
\end{remark}
\begin{proof} We begin with $\psi = \psi^\kappa_{+}$. 
By Lemma \ref{lem:3Sak.nG2}, to study stability for $\psi^\kappa_{+}$ we consider $\psi_+$ as in \eqref{eq:3Sak.G2str.psi} with
\begin{equation*}
    a=\frac{12}{5\kappa}+\frac{\kappa s}{12}A,\quad b=\frac{12}{5\kappa}+\frac{\kappa s}{12}B,\quad c=\frac{12}{\sqrt 5 \kappa}+\frac{\sqrt{5}\kappa s}{12}C 
\end{equation*}
for functions $A,B,C$ of $t$ and for $s$ small.  The reason for the $\sqrt{5}$ in the coefficient for $C$ in the expansion for $c$ is so that rescaling amongst the n$G_2$-structures where $a=b=\frac{1}{\sqrt{5}}c$ corresponds, at the linear level, to choosing $(A,B,C)$ in the direction of the vector $(1,1,1)$.  By expanding \eqref{eq:3Sak.nMLCF.0} in $s$ the linear terms in $s$ lead to the linearization of the flow \eqref{eq:3Sak.nMLCF} around the critical point $\psi^\kappa_{+}$.  We thus find that this linearized flow is given by
\begin{align*}
    \begin{pmatrix}
      \dot{A} \\
      \dot{B} \\
      \dot{C}
    \end{pmatrix} = \frac{5\kappa^2}{72}\begin{pmatrix}
      2(2-3\gamma)& 22-15\gamma & 4(3\gamma-2)\\
       2(22-15\gamma)& 9(\gamma-2)& 4(3\gamma-2)\\
     2(3\gamma-2)& 3\gamma-2& 6(4-3\gamma)
    \end{pmatrix} \begin{pmatrix}
      A\\
      B \\
      C
    \end{pmatrix} .
\end{align*}
One finds, by elementary linear algebra, that the negative and positive eigenspaces are spanned by
\begin{equation*}
    (A,B,C)\in\mathrm{Span}\{(-4,-4,3),(1,1,1)\} \quad \text{and}\quad  (A,B,C)\in\mathrm{Span}\{(-1,2,0)\}
\end{equation*}
respectively.
\begin{remark} Note that $(1,1,1)$ lying in the negative eigenspace is consistent with $\psi^\kappa_{+}$ being stable under rescaling.
\end{remark}
The statements about stability for the case $\psi^\kappa_{+}$ now follow immediately.  One also explicitly sees that the destabilizing eigendirection $\Psi_{+}$ is as claimed in \eqref{eq:3Sak.Psi+} and, using the Hodge star associated to $\psi^\kappa_{+}$, one finds that
\begin{equation}\label{eq:3Sak.starPsi+}
*\Psi_{+}=\frac{5\kappa}{12}\big(\eta_1\wedge\omega_1+\eta_2\wedge\omega_2-2\eta_3\wedge\omega_3\big) 
\end{equation}
which satisfies
\begin{equation*}
*\Psi_{+}\wedge\varphi^\kappa_{+}=0=*\Psi_{+}\wedge\psi^\kappa_{+}
\end{equation*}
We deduce that $\Psi_{+}\in\Omega^4_{27}$ as claimed. We see from \eqref{eq:3Sak.structure} that
\begin{equation}\label{eq:3Sak.Psi+.exact}
   \Psi_{+}=\tfrac{1}{4}d(2\eta_3\wedge\omega_3-\eta_1\wedge\omega_1-\eta_2\wedge\omega_2) 
\end{equation}
and so it is exact as desired. We may also compute using \eqref{eq:3Sak.starPsi+} and \eqref{eq:3Sak.Psi+.exact} that the first equation in \eqref{eq:3Sak:eigenvalues} holds.
%\begin{equation*}
%    d*\Psi_+=-\tfrac{5}{3}\kappa\Psi_+.
%\end{equation*}

Similarly, for $\psi = \psi_-^\kappa$ by Lemma \ref{lem:3Sak.nG2}, we consider $\psi_{-}$ as in \eqref{eq:3Sak.G2str.psi} with
\begin{equation*}
    a=\frac{4}{ \kappa}+\frac{\kappa s}{4}A,\quad b=\frac{4}{ \kappa}+\frac{\kappa s}{4}B,\quad c=\frac{4}{ \kappa}+\frac{ \kappa s}{4}C 
\end{equation*}
for functions $A,B,C$ of $t$ with $s$ small.  As before, expanding \eqref{eq:3Sak.nMLCF.0} and extracting the linear terms in $s$ leads to the linearized flow around the critical point $\psi_-^\kappa$. This linearized flow obtained this way is given by:
\begin{align*}
    \begin{pmatrix}
      \dot{A} \\
      \dot{B} \\
      \dot{C}
    \end{pmatrix} = \frac{\kappa^2}{8}\begin{pmatrix}
      10(2-3\gamma)& 5\gamma-2 & 4(5\gamma-2)\\
       2(5\gamma-2)& 5(\gamma-2)& 4(6-5\gamma)\\
     2(5\gamma-2)& 6-5\gamma& 2(4-5\gamma)
    \end{pmatrix} \begin{pmatrix}
      A\\
      B \\
      C
    \end{pmatrix}. 
\end{align*}
Again linear algebra yields that the negative and positive eigenspaces for the linearized flow are respectively spanned by
\begin{equation*}
  (A,B,C)\in\mathrm{Span}\{(-5,2,2),(1,1,1)\}\quad\text{and}\quad (A,B,C)\in\mathrm{Span}\{(0,-4,1)\}.  
\end{equation*}
The stability  results for $\psi_-^\kappa$ now follow directly.  Moreover, $\Psi_{-}$ is as claimed and its dual with respect to the Hodge star given by $\varphi_{-}^{\kappa}$ is given by
\begin{equation}\label{eq:3Sak.starPsi-}
*\Psi_{-}=\frac{\kappa}{4}(\eta_1\wedge\omega_1    +\eta_2\wedge\omega_2+\eta_3\wedge\omega_3-2\eta_1\wedge\eta_2\wedge\eta_3).
\end{equation}
One easily finds that
\begin{equation*}
    *\Psi_{-}\wedge\varphi_{-}^{\kappa}=0=*\Psi_{-}\wedge\psi_{-}^\kappa
\end{equation*}
and so $\Psi_{-}\in\Omega^4_{27}$.  We see again from \eqref{eq:3Sak.structure} that
\begin{equation}\label{eq:3Sak.Psi-.exact}
    \Psi_{-}=\tfrac{1}{6}d(2\eta_1\wedge\eta_2\wedge\eta_3-\eta_1\wedge\omega_1-\eta_2\wedge\omega_2-\eta_3\wedge\omega_3)
\end{equation}
and hence it is exact as required.  Finally, we may use \eqref{eq:3Sak.starPsi-} and \eqref{eq:3Sak.Psi-.exact} to obtain the second equation in \eqref{eq:3Sak:eigenvalues}.
 \end{proof}

\noindent Theorem \ref{thm:3Sak.instab.2} yields Theorem \ref{thm:3Sak.instab} in \S\ref{sec:intro}.
 
\begin{remark}
 We actually find that the other nearly parallel critical point $\psi_{\pm}^{(\gamma-1)\kappa}$ has index $2$ for the normalized modified $G_2$-Laplacian co-flow \eqref{eq:3Sak.nMLCF} within the family of co-closed $G_2$-structures given by \eqref{eq:3Sak.G2str.psi}, i.e.~it is not a source for the system, which would have index $3$.  Preliminary investigation suggests the presence of an additional critical point for the ODE system \eqref{eq:3Sak.nMLCF.0} which does \emph{not} correspond to an n$G_2$-structure.    
\end{remark}

\bibliographystyle{alpha}%unsrtnat 
\bibliography{Bibliografia-2024-04}% Produces the bibliography via BibTeX.

\end{document}